\documentclass[12pt]{article}
\usepackage{amsthm,amsmath,amsfonts,mathrsfs,amssymb,systeme,bm, bbm}
\usepackage{color,xcolor,graphicx}
\usepackage{natbib}
\usepackage{enumerate}
\usepackage[ruled,vlined]{algorithm2e}

\usepackage{mlmodern}
\usepackage{pdfrender,xcolor}
\usepackage{graphicx,color}
\usepackage[hidelinks,colorlinks=true,linkcolor=blue,citecolor=blue,urlcolor=blue]{hyperref}
\usepackage{algpseudocode}
\usepackage{comment}
\usepackage{cancel}
\usepackage{subcaption}

\newcommand{\R}{\mathbb{R}}

\newcommand{\Z}{\mathbb{Z}}
\newcommand{\E}{\mathbb{E}}
\newcommand{\Prob}{\mathbb{P}}

\newcommand{\Kdim}{{d}}

\newcommand{\I}[1]{\mathbbm{1}{\left( #1 \right)}} % indicator function
\newcommand{\T}{\prime} % \intercal
\newcommand{\up}[1]{{(#1)}}
\newcommand{\cc}{r} % center epsilon
\newcommand{\uu}{{\up{\cc}}} % upside epsilon
\newcommand{\fk}{{\up{\alpha}}} % upside epsilon
\newcommand{\inner}[1]{\langle #1 \rangle}
\newcommand{\Linner}[1]{\left\langle #1 \right\rangle}

\def\caP{\mathcal{P}}
\def\caM{\mathcal{M}}
\def\caS{\mathcal{S}}

\providecommand{\abs}[1]{\left\lvert#1\right\rvert}

\newtheorem{corollary}{Corollary}
\newtheorem{theorem}{Theorem}
\theoremstyle{definition}
\numberwithin{equation}{section}

\newtheorem{remark}{Remark}

\newtheorem{lemma}{Lemma}

\newtheorem{assumption}{Assumption}
\newtheorem{condition}{Condition}

\usepackage{graphicx}
\usepackage{tikz}
\usetikzlibrary{positioning}
\usetikzlibrary{calc,fit}
\usepgflibrary{shapes.geometric}
\definecolor{processblue}{cmyk}{0.96,0,0,0}

\tikzset{
	server/.style={circle, inner sep=0.0mm, minimum width=.8cm,
		draw=green,fill=green!10,thick},
	ellipseserver/.style={ellipse, inner sep=0.0mm, minimum width=.8cm,
		draw=green,fill=green!10,thick},
	buffer/.style={rectangle, rounded corners=3pt,
		inner sep=0.0mm, minimum width=.9cm, minimum height=.6cm,
		draw=orange,fill=blue!10,thick},
	vbuffer/.style={rectangle,rounded corners=3pt,
		inner sep=0.0mm,  minimum width=.6cm, minimum height=.8cm,
		draw=orange,fill=blue!10,thick},
	routing/.style={circle,inner sep=0pt,minimum size=.5cm,
		draw=red,minimum width=.1cm, fill=red!30},
	state/.style={inner sep=1.0mm, rounded corners=2pt,
		draw=green,fill=green!10,thick},
	dot/.style={circle, inner sep=0.0mm, minimum width=.6cm,
		draw=blue,fill=blue!10,thick},
	task/.style={circle, inner sep=0.0mm, minimum width=.6cm,
		draw=blue,fill=blue!10,thick},
	data/.style={rectangle, inner sep=0.0mm, minimum width=.5cm,
		minimum height=.4cm,
		draw=red,fill=red!10,thick}
}

\tikzset{
	pics/openrectangle/.style n args={3}{
		code = { %
			\pgfmathsetmacro\x{.5}
			\pgfmathsetmacro\y{4}
			\pgfmathsetmacro\z{.55}
			\coordinate (SN) at #1;
			\coordinate (SS) at ($(0,-.3)+#1$);
			\draw[thick,red] ([shift={(-\x,\y)}]SN)--([shift={(-\x, \z)}]SN)
			--([shift={(\x,\z)}]SN)--([shift={(\x,\y)}]SN);
			
			\foreach \i/\t in #2
			{
				\node[task] (\t) at ([shift={(0,\i)}]SN)  {\t};
			}
			
			\foreach \i/\d in #3
			{
				\node[data] (\d) at ([shift={(0,-\i)}]SS)  {C\d};
			}
		}
	}
}

\tikzset{
	pics/rack/.style n args={3}{
		code = { %
			\pgfmathsetmacro\x{1.7}
			\pgfmathsetmacro\y{4}
			\pgfmathsetmacro\z{-2}
			\pgfmathsetmacro\zz{-2.5}
			\coordinate (S) at ($#1!.5!#2$);
			\draw[thick,blue] ([shift={(-\x,\y)}]S)--([shift={(-\x, \z)}]S)
			--([shift={(\x,\z)}]S)--([shift={(\x,\y)}]S);
			\node at ([shift={(0,\zz)}]S) {Rack #3};
		}
	}
}

\title{Asymptotic Product-form Steady-state Distribution for Semimartingale Reflecting Brownian Motion in Multi-scaling Regime}

\author{Jin Guang$^1$ \and Xinyun Chen$^1$
\and J. G. Dai$^2$ \and Peter W. Glynn$^3$}

\date{{$^1$The Chinese University of Hong Kong, Shenzhen\\%
$^2$Cornell University}\\
$^3$Stanford University} 

\begin{document}

% \begin{titlepage}

% \tableofcontents

% \end{titlepage}
% \newpage

\maketitle

\begin{abstract}
	Inspired by \cite{DaiGlynXu2023}, we develop a novel multi-scaling asymptotic regime for semimartingale reflecting Brownian motion (SRBM). In this regime, we establish the steady-state convergence of SRBM to a product-form limit with exponentially distributed components by assuming the $\mathcal{P}$-reflection matrix and a uniform moment bound condition. We further demonstrate that the uniform moment bound condition holds in several subclasses of $\caP$-matrices. Our proof approach is rooted in the basic adjoint relationship (BAR) for SRBM proposed by \cite{HarrWill1987}.
\end{abstract}

\section{Introduction}

As a more parsimonious model, semimartingale reflecting Brownian motions (SRBMs) on the positive orthant are used to approximate the workload process for a wide range of stochastic networks in heavy traffic, including generalized Jackson network \citep{Reim1984} and multi-class queueing networks under various service principles \citep{HarrNguy1993, Bram1998a, Will1998a, ChenShenYao2002}. In particular, the steady-state distribution of SRBM provides a good approximation of the long-term performance of complex stochastic networks; see, for example, \citet{GamaZeev2006, BudhLee2009, Kats2010, ZhanZwar2008}.
However, computing the stationary distribution of the SRBM is challenging, except for a few special cases such as \citet{Fodd1983,HarrWill1987a,DiekMori2009}.
This difficulty has spurred the literature to develop numerical and simulation methods to compute the stationary distribution of SRBMs; see \citet{DaiHarr1992}, \cite{ShenChenDaiDai2002}, \citet{BlanChen2015}, \citet{BlanChenSiGlyn2021} and references therein.

In this paper, we take an alternative approach and establish an asymptotic approximation for the stationary distribution of SRBMs. The asymptotic regime we use is called the \textit{multi-scaling} regime and is inspired by \citet{DaiGlynXu2023}. Specifically, we consider a family of $d$-dimensional SRBM $\{Z^\uu(t), t\geq 0\}$ indexed by $r \in (0, 1)$ with data $(\Gamma, \mu^\uu, R)$, which are the covariance matrix, drift vector and reflection matrix of the $r$th SRBM, respectively. We denote by $\delta^\uu \equiv -R^{-1}\mu^\uu$ as the \textit{traffic slackness} vector of the $r$th SRBM. Under the multi-scaling regime, 
$$\delta^\uu_i=r^i, \quad \text{for }i=1,\ldots,d,$$
indicating that each dimension of the SRBM is scaled to ``heavy traffic" at different speeds. Under this multi-scaling regime, we prove that the stationary distribution of the scaled SRBM converges to a product form of independent exponential random variables for a variety of SRBM models, i.e.,
\begin{align*}%\label{eq: Z convergence}
		\left(\cc Z^{\uu}_1(\infty), \cc^2 Z^{\uu}_2(\infty),... ,  \cc^\Kdim Z^\uu_\Kdim(\infty)\right) \Rightarrow  (Z_1^*, \ldots, Z_\Kdim^*), \quad \text{as } \cc \to 0.
	\end{align*}
In addition, the mean of each $Z_k^*$ can be explicitly computed from data $\Gamma$ and $R$.

Our proof approach is rooted in the basic adjoint relationship (BAR) of SRBM proposed by \citet{HarrWill1987}. Specifically, BAR characterizes the stationary distribution of SRBM via a family of partial differential equations (PDEs), along with oblique derivative boundary conditions corresponding to each test function. In \citet{DaiHarr1992} and \cite{ShenChenDaiDai2002}, BAR serves as the key component in developing numerical methods for computing the stationary distribution of SRBM. In our setting, we use BAR to establish the convergence of the moment generating function (MGF) of the stationary distribution via a proper choice of test functions. 

The MGF convergence critically depends on a uniform moment bound condition. For an SRBM with $\caM$-reflection matrix, we manage to verify the uniform moment bound condition via mathematical induction built upon BAR with a class of carefully designed test functions. For an SRBM with a general $\caP$-reflection matrix, we show that the convergence of the MGF still holds, assuming the uniform moment bound condition. We then verify the uniform moment bound condition for two special examples: (1) SRBMs in two dimensions and (2) SRBMs with lower triangular reflection matrices. Whether such a uniform moment bound holds for general $\mathcal{P}$-matrices remains an open problem (see also \citet{DaiHuo2024} for the multi-class queueing network with SBP service discipline).

The rest of the paper is organized as follows: In Section \ref{sec: main results}, we present the main results in Theorem~\ref{thm: matrix P}. In Section \ref{sec: validation}, we validate the asymptotic results via the numerical experiments. The proof of Theorem \ref{thm: matrix P} is presented in Section \ref{sec: proof main}. In Section \ref{sec: uniform moment bound}, we present the proof of the uniform moment bounds for several subclasses of $\mathcal{P}$-reflection matrices, including the $\caM$-matrix, the two-dimensional case, and the lower triangular $\caP$-matrix.

\section{Model Setting and Main Results} \label{sec: main results}

\subsection{Model Setting and Assumptions}\label{subsec: assumption}

\noindent In this section, we present the definition of a semimartingale reflecting Brownian motion (SRBM) on the $d$-dimensional orthant $\R_+^d\equiv \{x\in \R^d:x\geq 0\}$, where $d$ is a positive integer. Let $\mu$ be a $d$-dimensional vector, $\Gamma$ be a $d\times d$ symmetric and positive definite matrix, and $R$ be a $d\times d$ matrix. 
An SRBM associated with $(\mu, \Gamma, R)$ is a continuous $\mathbb{F}$-adapted $d$-dimensional process $Z=\{Z(t),t\geq 0\}$, together with a family of probability measures $\{\Prob_x:x\in \R_+^d\}$, defined on some filtered probability space $(\Omega, \mathscr{F}, \mathbb{F}, \Prob_x)$, if it satisfies that under $\Prob_x$,
\begin{equation*}%\label{eq: RBM}
    Z(t) = X(t) + R Y(t) \in \R_+^d, \quad \text{for all } t\geq 0,
\end{equation*}
where $X=\{X(t),t\geq 0\}$ is a $d$-dimensional $\mathbb{F}$-Brownian motion with drift vector $\mu$, covariance matrix $\Gamma$ and initial condition $X(0)=x$ $\Prob_x$-almost surely, and $Y=\{Y(t),t\geq 0\}$ is an $\mathbb{F}$-adapted, $d$-dimensional process such that for each $i=1,\ldots,d$, the $i$th component $Y_i$ of $Y$ satisfies (i) $Y_i$ is continuous and nondecreasing with $Y_i(0)=0$, and (ii) $Y_i$ can increase only when $Z$ is on the face $F_i\equiv \{x\in \R^d_+:x_i=0\}$, i.e., $\int_0^t Z_i(s)dY_i(s)= 0$ for all $t\geq 0$. Hence, we have defined the so-called weak formulation of an SRBM associated with $(\mu, \Gamma, R)$ that starts from $x$.

\citet{ReimWill1988} and \citet{TaylWill1993} demonstrated that the necessary and sufficient conditions for the existence and uniqueness of $Z$ in law under $\Prob_x$ for each $x\in\R_+^d$ are that the reflection matrix $R$ is a completely-$\caS$ matrix. A square matrix $R$ is defined as an $\caS$-matrix if there exists a vector $u\geq 0$ such that $Ru>0$, and $R$ is completely-$\caS$ if every principal submatrix of $R$ is an $\caS$-matrix.

Roughly speaking, $Z$ behaves like a Brownian motion with initial state $x$, drift vector $\mu$ and covariance matrix $\Gamma$, and it is ``reflected'' or ``pushed'' instantaneously in the direction of the $i$th column of the reflection matrix $R$, denoted $R_{:,i}$, whenever it hits the boundary surface $F_i\equiv \{z\in \R_+^d:z_i=0\}$. The process $Y_i$ can be interpreted as the amount of time $Z_i$ spends on $F_i$, and is often called the \textit{regulator process} or \textit{pushing process}.

To ensure the positive recurrence of the SRBM, the drift vector $\mu$ and reflection matrix $R$ must satisfy the following necessary conditions in \citet{BramDaiHarr2010}:
\begin{equation} \label{eq1:necessary}
	\text{$R$ is nonsingular $\quad$ and  $\quad$  $\delta\equiv-R^{-1}\mu>0$}.
\end{equation} 
The vector $\delta$ can be understood as the proportion of time that the SRBM spends at the boundaries in the long run \citep{DaiDiek2011}, and we call it the \textit{traffic slackness} vector.

\vspace{5pt}

Inspired by the multi-scale heavy traffic regime studied in \citet{DaiGlynXu2023}, we propose a multi-scaling regime for SRBMs. Specifically, we consider a sequence of $d$-dimensional SRBMs $\{Z^\uu(t), t \geq 0\}$ indexed by $r\in (0,1)$. 
For each $r\in (0,1)$, the $r$th SRBM  has data $(\Gamma, \mu^\uu, R)$, indicating that the SRBMs share the same covariance and reflection matrices. Under the multi-scaling regime, we scale the traffic slackness vector as follows:

\begin{assumption}[Multi-scaling Regime] \label{assmpt: multiscale}
	For all $r\in (0,1)$,
	\begin{equation*}%\label{eq: drift epsilon}
		\delta^\uu \equiv -R^{-1}\mu^\uu=(r,~r^2,~\ldots,~r^d).
	\end{equation*}
\end{assumption}
Under Assumption \ref{assmpt: multiscale}, the traffic slackness in each dimension approaches 0 at different rates, which explains the name of the multi-scaling regime. Finally, as there is no sufficient and necessary condition to ensure the existence of stationary distribution, we just impose it as an assumption.

\begin{assumption}[Stability]\label{assmpt: stability}
	For each $r\in (0,1)$, the SRBM $\{Z^\uu(t), t \geq 0\}$ is positive recurrent and has a unique stationary distribution $\pi^\uu$ on $\R_+^d$.
\end{assumption}

We denote by $Z^{(r)}$ the random vector in $\R_+^d$ that follows the stationary distribution of the $r$th SRBM. 
To shorten the notation, we use $\E_{\pi}[\cdot]$ (rather than $\E_{\pi^\uu}[\cdot]$) to denote expectation with respect to the stationary distribution when the index $r$ is clear from the context.

\begin{assumption}[Reflection Matrix]\label{assmpt: reflect matrix}
    The reflection matrix $R$ is a $\caP$-matrix.
\end{assumption}

A square matrix is a $\caP$-matrix if all of its principal minors are positive. It is an $\caM$-matrix if it can be expressed as $sI-B$, where $B$ is a nonnegative matrix with the spectral radius less than $s$. The class of $\caM$-matrices is a subclass of $\caP$-matrices, and the class of $\caP$-matrices is a subclass of completely-$\caS$ matrices. 

We further introduce a uniform moment bound condition, which allows us to develop a unified proof framework for general SRBMs.

\begin{condition}[Uniform Moment Bounds]\label{assmpt: moment bound P}
	We assume that for any $i,k\in \{1,2,...,d\}$,
	\begin{align} \label{eq: moment bound}
		\E_{\pi} \left[ \left( r^k Z_k^\uu \right)^d \right] =O(1), \quad 
		\E_{\pi}\left[\int_0^1 \left( r^{k}Z_k^\uu(t) \right)^d d Y_i^\uu(t)\right] = o(r^{i-1}) \quad \text{as } r \to 0,
	\end{align}
	where adopt the notation $f(r)=O(g(r))$ to indicate that there exist constants $r_1>0$ and $C>0$ such that $|f(r)| \leq C|g(r)|$ for all $r\in (0,r_1)$. Moreover, $f(r)=o(g(r))$ represents that $f(r) / g(r) \rightarrow 0$ as $r \downarrow 0$.
\end{condition}

The uniform moment bound condition ensures that, for any \(\varepsilon_0 > 0\) and $k=1,\ldots,d$, \( r^{k+\varepsilon_0} Z_k^\uu \) converges to zero in probability as \( r \to 0 \), in both interior and boundaries of the state space.

The necessity of imposing such high-order (\(d\)-th order) moment conditions arises from the truncation techniques employed when dealing with unbounded test functions in Section \ref{subsec: proof sketch P}. In the special case of $\caM$-reflection matrices, we demonstrate that Condition \ref{assmpt: moment bound P} can be weakened by replacing $d$ in \eqref{eq: moment bound} with a constant $2$, as shown in Section \ref{subsec: proof sketch}.

Later in Corollaries \ref{prop: M matrix}-\ref{prop: FBFS reentrant Queue} of the next section, we will demonstrate that Condition \ref{assmpt: moment bound P} is satisfied for some reflection matrices, including $\caM$-matrices, two-dimensional $\caP$-matrices, and lower triangular $\caP$-matrices.
\subsection{Main Results}

\noindent In this section, we establish the steady-state convergence of the scaled SRBM to independent exponential random variables in the multi-scaling regime. 
\begin{theorem} \label{thm: matrix P} 
	Suppose that Assumptions \ref{assmpt: multiscale}-\ref{assmpt: reflect matrix} and Condition \ref{assmpt: moment bound P} hold. Then
	\begin{align}\label{eq: Z convergence P}
		\left(\cc Z^{\uu}_1, \cc^2 Z^{\uu}_2,... ,  \cc^\Kdim Z^\uu_\Kdim\right) \Rightarrow  (Z_1^*, \ldots, Z_\Kdim^*), \quad \text{as } \cc \to 0.
	\end{align}
	Furthermore, $Z_1^*, \ldots, Z_\Kdim^*$ are independent, and for each component $k$, $Z_k^*$ is an exponential random variable with mean
	 \begin{equation}\label{eq: mean k}
		m_k = \frac{u_k^\T\Gamma u_k}{2u_k^\T R_{:,k}} > 0,
	\end{equation}
	where $^\T$ denotes the transpose, $R_{:,k}$ is the $k$th column of $R$, and $u_k\in \R^d$ is defined as
	\begin{equation}\label{eq: u}
		u_k=[w_{1k},\cdots, w_{k-1,k}, 1, 0, \ldots, 0]^\T\in \R^d,
	\end{equation}
	where the first $k-1$ elements of $u_k$ are the unique solution to the system of equations, with the uniqueness following from the $\caP$-matrix property of $R$:
	\begin{equation}\label{eq: w}
		\sum_{j=1}^{k-1}w_{jk} R_{j\ell} + R_{k\ell} = 0\quad \text{ for } \ell=1,\ldots,k-1,
	\end{equation}
	such that $u_k'R_{:,\ell}=0$ for all $\ell<k$.
\end{theorem}

\begin{remark}[Comparison to skew symmetric case] \label{rmk: skew}
	\citet{HarrWill1987} demonstrate that if the SRBM $Z$ satisfies the skew symmetric condition, i.e., 
	\begin{equation} \label{eq: skew symmetric}
		2\Lambda \Gamma \Lambda = \Lambda R V \Lambda^{-1} + \Lambda^{-1}V R^{\T} \Lambda,
	\end{equation}
	where $\Lambda = \operatorname{diag}(\Gamma_{11}^{-1/2}, \Gamma_{22}^{-1/2}, \ldots, \Gamma_{dd}^{-1/2})$ and $V=\operatorname{diag}(R_{11}^{-1}, R_{22}^{-1}, \ldots, R_{dd}^{-1})$,
	then the SRBM $Z$ has a product-form exponential stationary distribution with mean $\Gamma_{kk}/(2R_{kk}\delta_k)$ for $k$th component. If the family of SRBMs in our setting satisfies the skew symmetric condition, then Theorem \ref{thm: matrix P} recovers the result in the skew symmetric case. The proof is provided in Appendix \ref{subsec: proof lemmas}.
\end{remark}

\begin{remark}[Interpretation of $w$] \label{rmk: w}
	When SRBM serves as the diffusion approximation of an open generalized Jackson network, the reflection matrix is $R=I - P^\T$ (see \citet{Reim1984}), which is an $\caM$-matrix.
	Here $P$ represents the routing matrix of the network, and for all $i, j \in \{1, 2, \ldots, d\}$, $P_{ij}$ denotes the probability that a customer is routed to station $j$ after service completion at station $i$. In this context, the value of $w_{ij}$ is the probability that a customer starting from station $i$ will enter station $j$ before exiting the network or visiting any stations in $\{j + 1, ..., d\}$, as described in Lemma 7.3 of \citet{DaiGlynXu2023}. In this case, all elements of $u_k$ are nonnegative. However, for non-$\caM$-matrices, the signs of the elements are not guaranteed to be nonnegative. 
\end{remark}

When the reflection matrix $R$ is an $\caM$-matrix, \citet{HarrWill1987} demonstrated that condition \eqref{eq1:necessary} is also sufficient for the positive recurrence of the SRBM. Consequently, Assumption~\ref{assmpt: multiscale} indicates that Assumption \ref{assmpt: stability} holds. Besides, Condition \ref{assmpt: moment bound P} also holds for $\caM$-matrices, as shown in Section~\ref{subsec: uniform moment bound}. Therefore, we obtain the following corollary.

\begin{corollary}[$\caM$-reflection matrix Case]\label{prop: M matrix}
	Suppose that $R$ is a $\caM$-matrix and that Assumption \ref{assmpt: multiscale} holds. Then, conclusion \eqref{eq: Z convergence P} holds.
\end{corollary}

When $R$ is a non-$\caM$-matrix, analogous corollaries hold for several subclasses of $\caP$-matrices. The following two corollaries address two prominent non-$\caM$-matrix cases: the two-dimensional $\caP$-matrix and the lower triangular $\caP$-matrix. In both cases, positive recurrence of the SRBM under condition \eqref{eq1:necessary} has been established in prior literature; see \citet{HobsRoge1993,HarrHase2009} and \citet{Chen1996a}, respectively. Condition \ref{assmpt: moment bound P} is also satisfied, whose proofs are provided in Sections~\ref{subsubsec: 2d P matrix moment} and \ref{subsec: FBFS P matrix moment}.

\begin{corollary}[Two-dimensional Case]\label{prop: 2-d P matrix}
	Suppose $d=2$ and that Assumptions \ref{assmpt: multiscale} and \ref{assmpt: reflect matrix} hold. Then, conclusion \eqref{eq: Z convergence P} holds.
\end{corollary}

\begin{corollary}[Lower Triangular Case]\label{prop: FBFS reentrant Queue}
	Suppose that the reflection matrix $R$ is a lower triangular matrix, i.e., $R_{ij}=0$ for $i<j$, and that Assumption \ref{assmpt: multiscale} and \ref{assmpt: reflect matrix} hold. Then, conclusion \eqref{eq: Z convergence P} holds. 
\end{corollary}

\begin{remark}
	Corollary \ref{prop: FBFS reentrant Queue} encompasses several interesting cases within the diffusion limit of queueing networks, including (i) multi-class feed-forward queueing networks, as described in \citet{HarrWill1992}, where customers always flow from lower-numbered stations to higher numbered ones; and (ii) multi-class re-entrant queueing networks operating under a first-buffer-first-served discipline, detailed in \citet{DaiYehZhou1997}, where stations are ordered inversely by their last visit time, assigning the highest number to the most recently visited station.
\end{remark}

\section{Numerical Experiments} \label{sec: validation}

This section presents numerical experiments of the asymptotic approximation derived in Theorem \ref{thm: matrix P} by two-dimensional examples. In Section \ref{subsec: analytical validation}, we demonstrate the convergence and its rate of the multi-scaling approximation for the steady-state average and correlation between two dimensions. In Section \ref{subsec: skew symmetric comparison}, we compare the multi-scaling approximation to the skew symmetric approximation, which is an approximation in the literature. Finally, in Section~\ref{subsec: simulation}, we show that the approximation helps improve the efficiency of simulation algorithms.

\subsection{Numerical Validation of Asymptotic Approximation} \label{subsec: analytical validation}

To validate the convergence and estimate the convergence rate of the multi-scaling approximation, we conduct numerical experiments on a two-dimensional example for which the exact steady-state moments are analytically available. This allows for a direct and quantitative comparison between the asymptotic results and the true values.

We consider the SRBM defined by the following parameters:
\begin{equation*}
	\Gamma = \begin{bmatrix} 1 & 0 \\ 0 & 1 \end{bmatrix}, \quad R = \begin{bmatrix} 1 & -\beta \\ 0 & 1 \end{bmatrix}, \quad \delta^\uu = \begin{bmatrix} r \\ r^2 \end{bmatrix}, \text{ and } \mu^\uu = -R\delta^\uu = \begin{bmatrix} -r + \beta r^2 \\ -r^2 \end{bmatrix},
\end{equation*}
where $\beta \in [0, 1)$ and $r \in (0, 1)$. The exact steady-state mean and correlation for this setting are given analytically by \citet{Fodd1983}, providing a rigorous benchmark for our numerical validation.

According to Theorem~\ref{thm: matrix P}, the scaled stationary vector $(r Z_1^{(r)}, r^2 Z_2^{(r)})$ converges weakly to $(Z_1^\ast, Z_2^\ast)$ as $r \to 0$, where $Z_1^\ast$ and $Z_2^\ast$ are independent exponential random variables with means $m_1 = m_2 = 0.5$. Notably, the theoretical scaled mean of the second dimension, $\E[r^2 Z_2^{(r)}]$, is exactly 0.5 for all $r \in (0,1)$ and $\beta \in [0,1)$, perfectly matching the asymptotic limit. Therefore, our numerical validation focuses on the relative errors of the scaled mean of the first dimension, $\E[r Z_1^{(r)}]$, towards 0.5, and the correlation between the two scaled dimensions, $r Z_1^{(r)}$ and $r^2 Z_2^{(r)}$, towards zero.

\begin{figure}[htb]
    \begin{subfigure}[b]{0.48\textwidth}
        \centering
        \includegraphics[width=\textwidth]{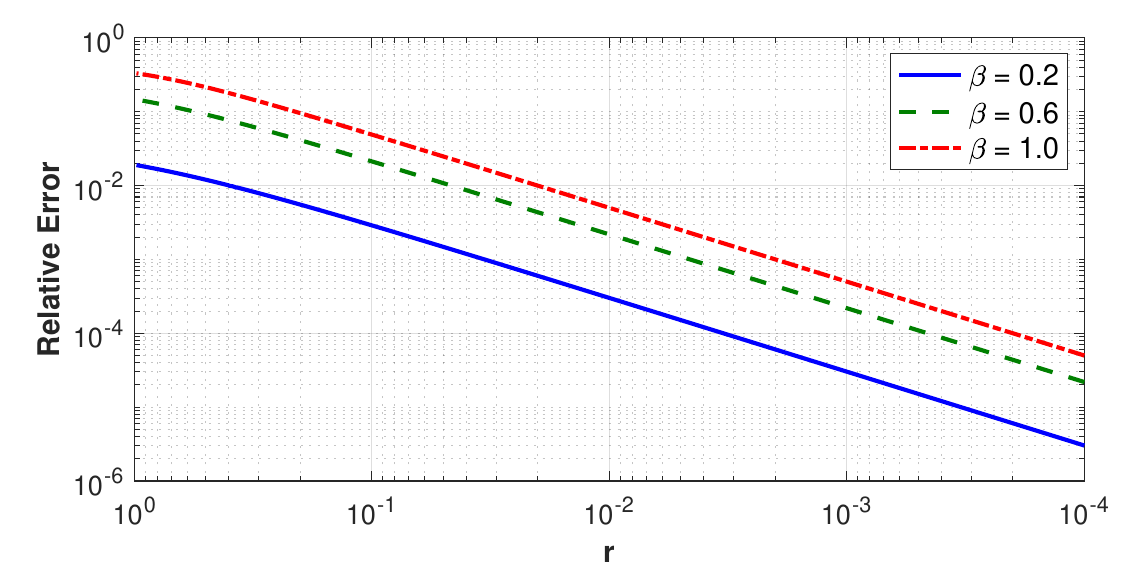}
        % Assuming this figure corresponds to panel (c) in the caption
    \end{subfigure}
    \hfill
    \begin{subfigure}[b]{0.48\textwidth}
        \centering
        \includegraphics[width=\textwidth]{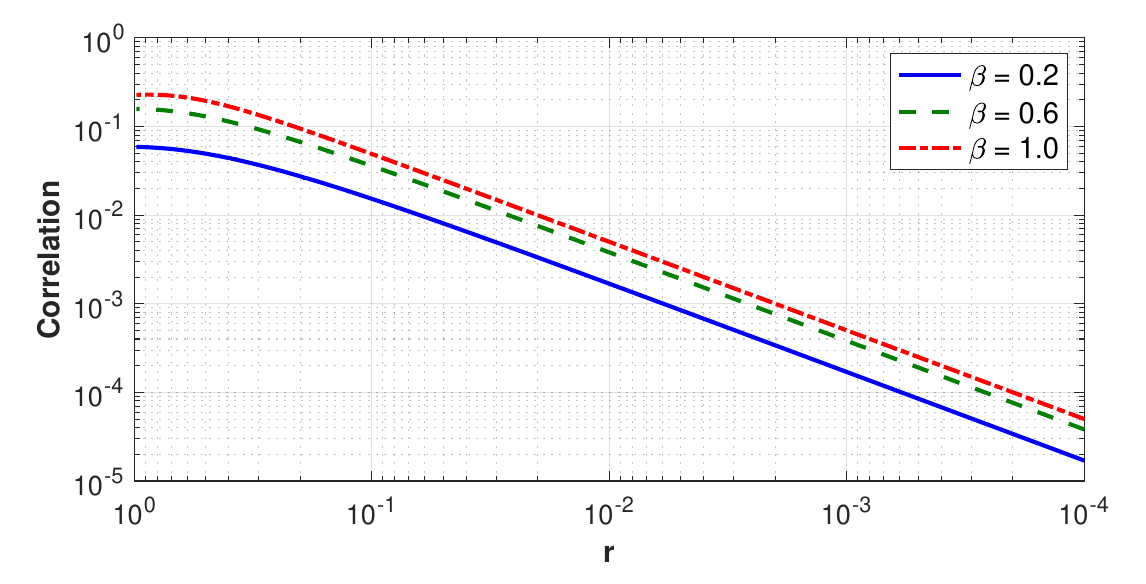}
        % Assuming this figure corresponds to panel (d) in the caption
    \end{subfigure}
    
    \caption{
		Numerical validation of the asymptotic approximation. 
		Left: Log-log plot of the relative error in the mean of the first dimension;
		Right: Log-log plot of the correlation between two dimensions.
		}
    \label{fig:validation_results}
\end{figure}

Figure~\ref{fig:validation_results} summarizes the main results. The left panel displays a log-log plot of the relative error for the mean of the first dimension, defined as $|\E[r Z_1^{(r)}] - 0.5| / \E[r Z_1^{(r)}]$. This plot demonstrates that the error decreases linearly with $r$, particularly in the range $r \in [10^{-4}, 10^{-1}]$. A regression performed on the log-log data (comprising 100 points) yields a slope ranging from 0.9959 to 0.9981 for various $\beta$ values, thus validating the convergence rate $O(r)$. Furthermore, we observe that smaller values of $\beta$ lead to lower relative errors with weaker inter-dimensional dependency. The right panel presents a log-log plot of the correlation between $r Z_1^{(r)}$ and $r^2 Z_2^{(r)}$. This plot is consistent with the asymptotic independence in Theorem~\ref{thm: matrix P} and exhibits a similar linear convergence towards zero.

In summary, the numerical results provide strong evidence for the theoretical convergence established in Theorem~\ref{thm: matrix P}. The observed linear convergence rate is an interesting point that deserves more detailed study in the future.

\subsection{Comparison to Skew Symmetric Approximation} \label{subsec: skew symmetric comparison}
In this section, we compare our multi-scaling approximation to the skew symmetric approximation in specific two-dimensional cases, where the exact stationary distribution is available analytically by \citet{Fodd1983}. Although many SRBMs do not satisfy the skew symmetric condition \eqref{eq: skew symmetric}, the skew symmetric approximation remains widely used in the literature. This approximation ignores the condition and utilizes a product-form distribution to approximate the stationary distribution, with means $m_k =\Gamma_{kk}/(2R_{kk}\delta_k)$ for $k=1,\ldots,d$ in Remark \ref{rmk: skew}.

We consider the SRBM defined by the following parameters:
\begin{equation*}
	\Gamma = \begin{bmatrix} 1 & 0 \\ 0 & 1 \end{bmatrix}, \quad R = \begin{bmatrix} 1 & -\beta \\ -\alpha & 1 \end{bmatrix}, \quad \delta^\uu = \begin{bmatrix} r \\ r^2 \end{bmatrix}, \text{ and } \mu^\uu = -R\delta^\uu = \begin{bmatrix} -r + \beta r^2 \\ \alpha r - r^2 \end{bmatrix},
\end{equation*}
where $\alpha,\beta \in (0, 1)$ and $r \in (\alpha, 1)$ to ensure $\mu^\uu_1 < 0$ and $\mu^\uu_2 < 0$ so that the exact steady-state mean for this setting is given analytically by \citet{Fodd1983}. The skew symmetric approximation suggests that $\E[r Z_1^{(r)}]=\E[r^2 Z_2^{(r)}] = 0.5$ for all $r \in (\alpha, 1)$.

As this setting violates the skew symmetric condition \eqref{eq: skew symmetric} and is designed to be in multi-scaling regime, we expect the multi-scaling approximation to outperform the skew symmetric approximation. According to Theorem~\ref{thm: matrix P}, the scaled stationary vector $(r Z_1^{(r)}, r^2 Z_2^{(r)})$ converges weakly to $(Z_1^\ast, Z_2^\ast)$ as $r \to 0$, where $Z_1^\ast$ and $Z_2^\ast$ are independent exponential random variables with means $0.5$ and $(1+\alpha^2)/(2(1-\alpha\beta))$. Notably, two approximations have the same mean of the first dimension (this holds for any multi-dimensional SRBM, since $u_1=[1,0,\ldots,0]^\T$ and then $m_1 = \Gamma_{11}/(2R_{11})$, which is the same as the mean of the first dimension in the skew symmetric approximation as shown in Remark \ref{rmk: skew}). Therefore, our numerical comparison focuses on the scaled mean of the second dimension, $\E[r^2 Z_2^{(r)}]$.

\begin{figure}[htb]
	\begin{subfigure}[b]{0.48\textwidth}
        \centering
		{\tiny (a) $\alpha=0.25$, $\beta=0.5$}
        \includegraphics[width=\textwidth]{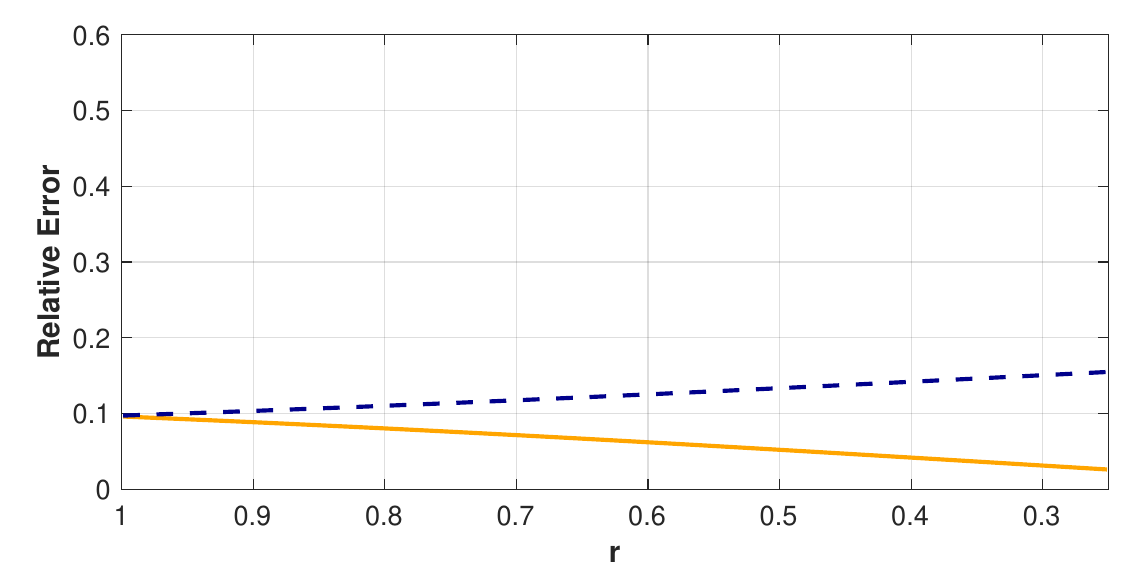}
        % Assuming this figure corresponds to panel (c) in the caption
    \end{subfigure}
    \hfill
    \begin{subfigure}[b]{0.48\textwidth}
        \centering
		{\tiny (b) $\alpha=0.5$, $\beta=0.5$}
        \includegraphics[width=\textwidth]{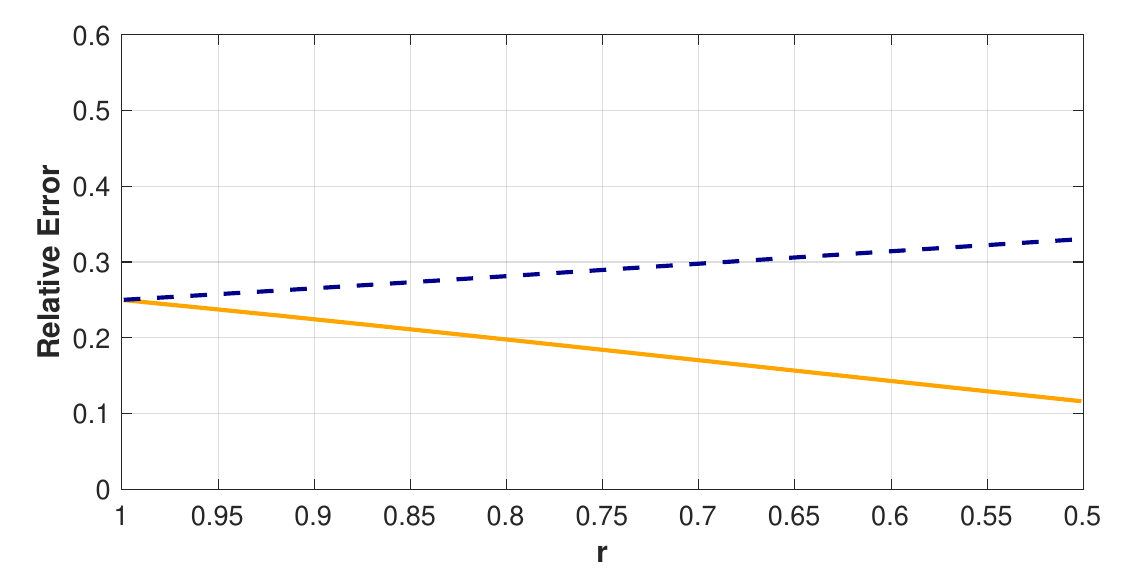}
        % Assuming this figure corresponds to panel (d) in the caption
    \end{subfigure}

	\begin{subfigure}[b]{0.48\textwidth}
        \centering
		{\tiny (c) $\alpha=0.25$, $\beta=1$}
        \includegraphics[width=\textwidth]{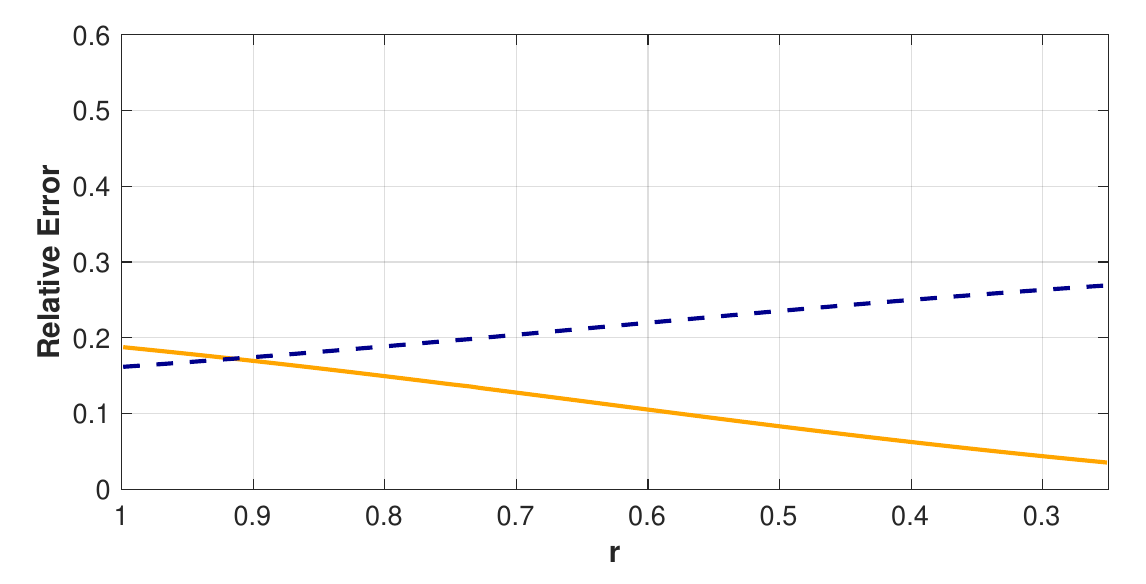}
        % Assuming this figure corresponds to panel (c) in the caption
    \end{subfigure}
    \hfill
    \begin{subfigure}[b]{0.48\textwidth}
        \centering
		{\tiny (d) $\alpha=0.5$, $\beta=1$}
        \includegraphics[width=\textwidth]{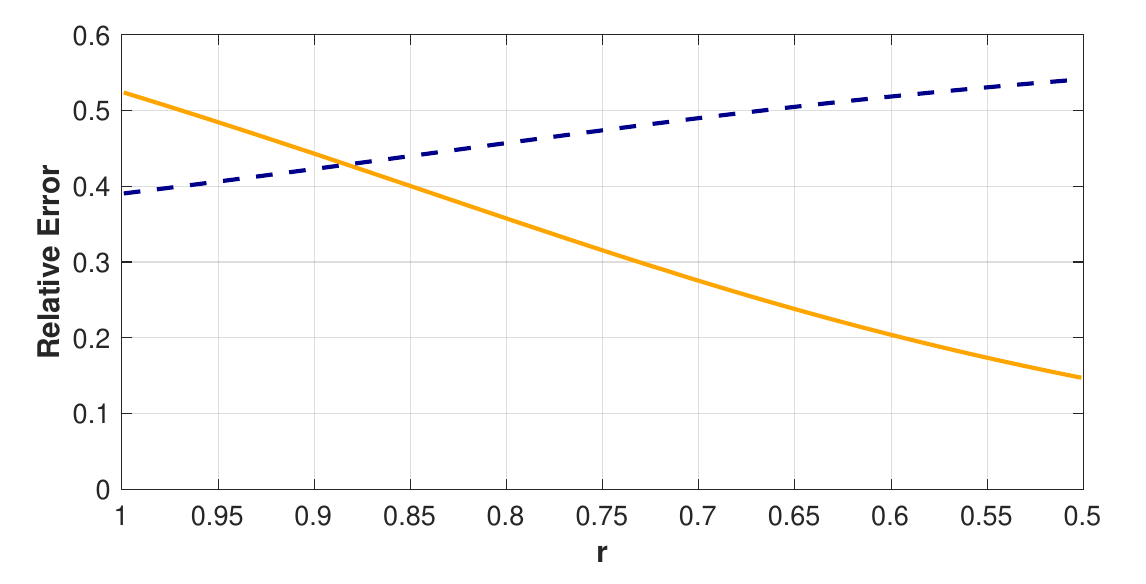}
        % Assuming this figure corresponds to panel (d) in the caption
    \end{subfigure}
    
    \caption{
		Relative errors in $\E[r^2 Z_2^{(r)}]$ of the multi-scaling approximation (solid line) and the skew symmetric approximation (dashed line) with $\alpha\in\{0.25, 0.5\}$ and $\beta\in\{0.5, 1\}$.
		}
    \label{fig:skew_symmetric_comparison}
\end{figure}

Figure~\ref{fig:skew_symmetric_comparison} presents the relative errors in estimating \(\E[r^2 Z_2^{(r)}]\) for both the multi-scaling approximation and the skew symmetric approximation, benchmarked against the exact value, over the range \(r \in (\alpha, 1)\) with \(\alpha \in \{0.25, 0.5\}\) and \(\beta \in \{0.5, 1\}\) for example. Across the four panels, we clearly observe that as \(r\) decreases, the relative error of the multi-scaling approximation consistently reduces, whereas the error of the skew symmetric approximation increases substantially. Comparing Panels (a) and (c), we observe that for values of \(r\) close to 1 and \(\beta = 0.5\), the two approximations produce similar errors. However, at a larger \(\beta = 1\), as shown in Panels (b) and (d), the skew symmetric approximation has a slight advantage for higher \(r\), whereas the multi-scaling approximation exhibits significantly better performance when \(r\) falls below 0.88. 

Further analysis of the error magnitudes across all four panels indicates that both approximations achieve higher accuracy for smaller values of \(\beta\) or \(\alpha\). This observation aligns with the findings in Section~\ref{subsec: analytical validation} for the multi-scaling approximation. For the skew symmetric approximation, this  results from being closer to satisfying the skew symmetric condition.

In summary, the numerical experiments reveal that the multi-scaling approximation outperforms the skew symmetric approximation, especially in the regime of small $r$.

\subsection{Enhancing Simulation Algorithms via Approximation} \label{subsec: simulation}
Recently, \citet{BlanChenSiGlyn2021} introduced an efficient multi-level Monte Carlo (MLMC) method for estimating the stationary mean of SRBMs. The efficiency of this method, however, depends on the SRBM satisfying a uniform stability assumption—namely, that the slackness vector $\delta^\uu_k > \varepsilon$ for all $k=1,\ldots,d$ and some $\varepsilon > 0$. In our setting, this condition is not met, which can lead to slow simulation towards the steady state. Given the significant impact of the initial distribution on the performance of the simulation, we propose employing our multi-scaling approximation as the initial distribution to accelerate the MLMC simulations.

% Setup
We consider a three-dimensional SRBM with the following parameters:
\begin{equation*}
	\Gamma = \begin{bmatrix} 1 & 0 & 0 \\ 0 & 1 & 0 \\ 0 & 0 & 1 \end{bmatrix}, \quad R = \begin{bmatrix}
		1 & -0.6 & -0.4 \\
		-0.5 & 1 & -0.4 \\
		-0.2 & -0.3 & 1
	\end{bmatrix}, \quad \delta^\uu = \begin{bmatrix} r \\ r^2 \\ r^3 \end{bmatrix},
\end{equation*}
and $\mu^\uu = -R\delta^\uu$, where the scaling parameter $r$ is set to $0.2$.

% Procedure
We examine the impact of three distinct initial distributions on simulation performance: (1) origin, where the initial state is fixed at \((0,0,0)\); (2) skew symmetric approximation, where the initial state follows a product of independent exponential distributions with mean vector \((2.5,12.5,62.5)\); and (3) multi-scaling approximation, where the initial state is similarly drawn from independent exponential distributions but with mean vector \((2.5,22.32, 179.69)\).

We employ the MLMC method with the above initial distributions to estimate the stationary means \(\E[Z_1^{(0.2)}]\), \(\E[Z_2^{(0.2)}]\) and \(\E[Z_3^{(0.2)}]\). For each initial distribution, we generate 50 independent MLMC estimators to compute the corresponding 95\% confidence intervals, where each estimator is constructed from 2,000 independent sample paths. Since our setting does not satisfy the uniform stability assumption required by the standard MLMC method, the default hyperparameters may not yield efficient results. Therefore, we systematically explore a range of MLMC hyperparameter configurations, varying the number of levels \(L \in \{2,4,6,8\}\), the base simulation horizon \(T \in \{5000, 100000\}\), and the discretization base \(\gamma \in \{0.05, 0.1, 0.2\}\). Specifically, for each sample path simulated at level \(l \in \{0,\ldots,L-1\}\), the simulation horizon is set to \((l+1)T\), and the discretization step size is chosen as \(100\gamma^l\). The computational complexity is measured by the total number of discretization steps performed.

\begin{figure}[!htb]
	\centering
	\includegraphics[width=\textwidth]{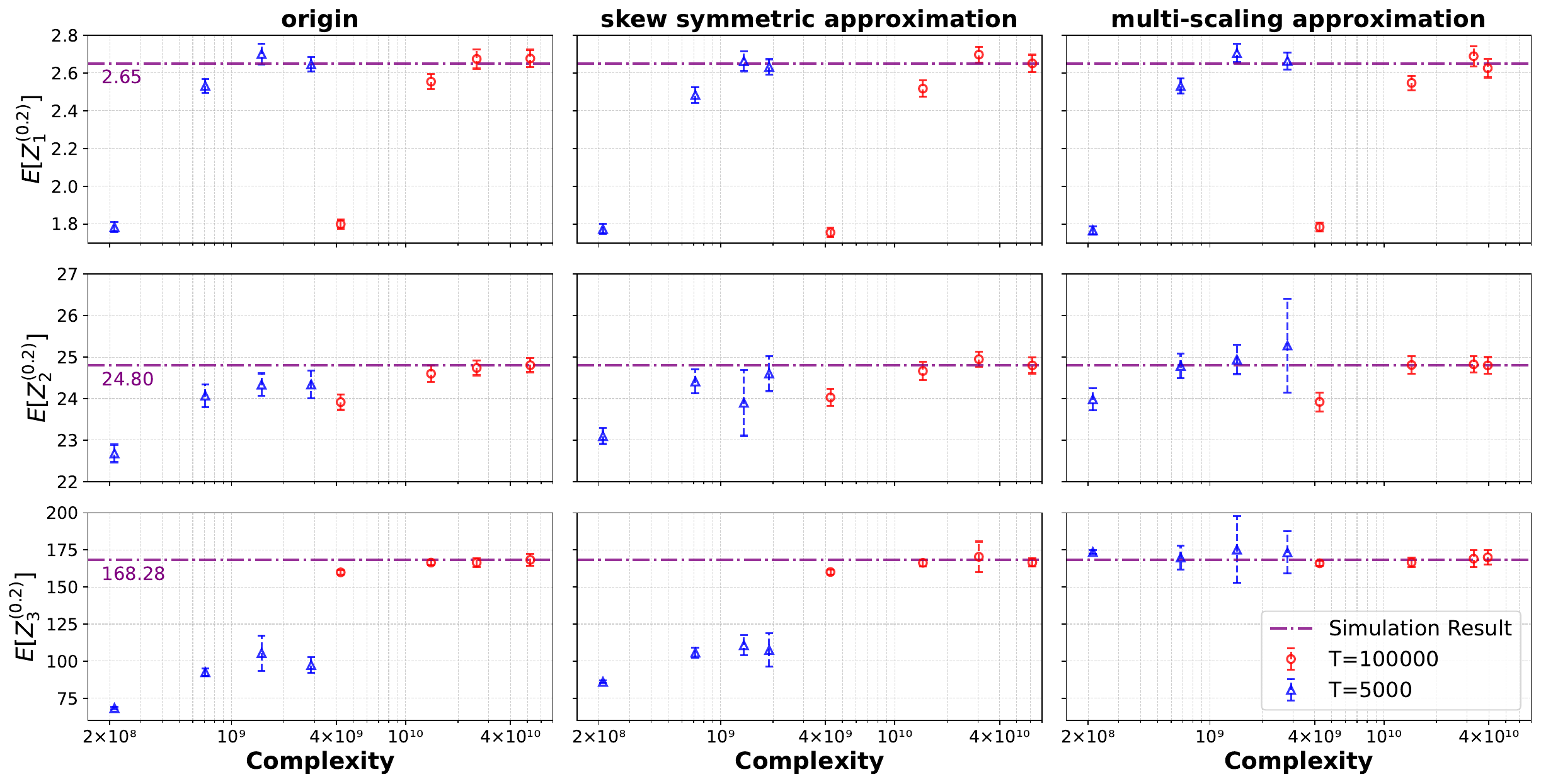} 
    
    \caption{
        MLMC simulation results for different initial distributions. Columns (left to right) correspond to initial distributions: origin, skew symmetric approximation, and multi-scaling approximation. Panels of row $k$ demonstrate estimated means $\E[Z_k^{(0.2)}]$ for $k=1, 2,3$ against computational cost (log scale). Dash-dotted line: estimated mean by averaging the simulation results with the highest computational cost within each panel in a row. 
    }
    \label{fig:sim1}
\end{figure}

% Results & Figure Explanation
The numerical results are summarized in Figure \ref{fig:sim1}. The figure is structured into three columns, each representing a distinct initial distribution strategy: the origin (left column), skew symmetric (middle column), and multi-scaling (right column). Each column further comprises three rows, corresponding respectively to the estimated stationary means \(\E[Z_1^{(0.2)}]\), \(\E[Z_2^{(0.2)}]\) and \(\E[Z_3^{(0.2)}]\). Within each panel, the estimated means are plotted against computational cost, presented on a logarithmic scale. Each data point, accompanied by an error bar, denotes an MLMC estimate along with its associated 95\% confidence interval, obtained under a specific configuration of MLMC hyperparameters (\(L, T\)). Points that share the same color (each color also corresponds to a unique shape) within each panel indicate simulations conducted with an identical base simulation horizon \(T\), but varying numbers of MLMC levels \(L\). Although we experimented with \(\gamma \in \{0.05, 0.1, 0.2\}\), the resulting patterns were qualitatively similar across these values. Therefore, for clarity and conciseness, we present only the results of \(\gamma=0.2\). 
Since the true stationary means are not analytically available, we provide an estimated reference value, represented by a dotted line. The estimate for each row is obtained by averaging the simulation results with the highest computational cost within each panel in that row. 

Comparing the three panels in the first row, corresponding to the less heavily loaded dimension, we observe that the discretization error constitutes the primary source of bias. This observation is supported by the consistent reduction in bias as the number of MLMC levels \(L\) increases, indicated by points of the same color. Notably, neither the base simulation horizon \(T\) nor the choice of initial distribution significantly affects this discretization bias.

In contrast, the third row, representing the more heavily loaded dimension, reveals that the transient error dominates the bias. Specifically, the first panel of this row clearly demonstrates a substantial bias when employing shorter base simulation horizons \(T\) (see the blue triangle points). When we utilize the skew symmetric approximation as the initial distribution, the bias is slightly reduced as shown in the second panel. However, adopting our multi-scaling approximation in the third panel significantly reduces this transient bias. This reduction underscores the significant effectiveness of our proposed approximation in improving the efficiency of SRBM simulations for short simulation horizons. In this case, using the multi-scaling approximation as the initial distribution takes less than 10\% of the computational cost to reach the same level of precision compared to other initial distributions. The second row has a similar pattern combined with bias from discretization error and transient error.

% Conclusion
In conclusion, the experimental findings demonstrate that employing the multi-scaling approximation as the initial distribution substantially enhances the efficiency of SRBM simulations. This improvement is especially pronounced for dimensions operating near heavy traffic and with short simulation horizons.

\section{Proof of Main Results} \label{sec: proof main}

Our proof approach is rooted in the basic adjoint relationship (BAR),  first proposed by  \cite{HarrWill1987}. This framework allows us to directly characterize the stationary distribution of the SRBMs, thereby eliminating the need to analyze their transient dynamics. Specifically, we are able to employ the BAR with appropriately designed test functions to establish desired results.

In Section \ref{subsec: BAR}, we provide a brief overview of the BAR for SRBMs. Since the signs of the elements of $u_k$ in \eqref{eq: u} are not guaranteed to be nonnegative for a $\caP$-reflection matrix, the design of the test function is somewhat more complicated. To illustrate the main ideas in the proof of Theorem \ref{thm: matrix P}, we first present the proof for Corollary \ref{prop: M matrix}, the case of the $\caM$-reflection matrix, in Section \ref{subsec: proof sketch}. In Section \ref{subsec: proof sketch P}, we introduce a truncation technique to overcome the sign issue and complete the proof of Theorem \ref{thm: matrix P}.

\subsection{Basic Adjoint Relationship} \label{subsec: BAR}

\cite{HarrWill1987} and \citet{DaiDiek2011} demonstrate that the stationary distribution $\pi^\uu$ of the SRBM $\{Z^\uu(t), t \geq 0\}$ for $r\in (0,1)$ must satisfy the so-called \textit{basic adjoint relationship} (BAR), as follows.

\begin{lemma}[Basic Adjoint Relationship]
	For any $r\in (0,1)$ and each $i=1,2,...,d$, there is a probability measure $\nu_i^\uu$ concentrating on the boundary surface $F_i$, which is absolutely continuous with respect to the Lebesgue measure on $F_i$. This ensures that for any bounded Borel function $f$,
	\begin{equation} \label{eq: nu}
		\E_{\pi^\uu}\left[\int_0^1 f(Z^\uu(s))d Y_i^\uu(s)\right]= \delta_i^\uu \int_{F_i}fd\nu_i^\uu.
	\end{equation}
	Moreover, for any $f\in C_b^2(\R_+^d)$,
	\begin{equation}\label{eq:BAR87}
		\int_{\R_+^d}\left( \frac{1}{2} \inner{\Gamma, \nabla^2 f} + \inner{\mu^\uu,  \nabla f} \right)  d \pi^\uu+  \sum_{i=1}^\Kdim \delta_i^\uu \int_{F_i} \inner{R_{:,i},  \nabla  f} d \nu_i^\uu=0,
	\end{equation}
where $C_b^2(\R_+^d)$ is the set of twice continuously differentiable functions $f$ on $\R_+^d$ such that $f$ and its first and second-order derivatives are all bounded, $R_{:,i}$ is the $i$th column of $R$, and $\inner{}$ denotes the Frobenius inner product, defined as $\inner{A,B}=\sum_{i,j}A_{ij}B_{ij}$ for matrices $A$ and $B$, and $\inner{a,b}=\sum_{i}a_ib_i$ for vectors $a$ and $b$. 
\end{lemma}
The probability measures $\{\nu_i^\uu\}_{i=1}^d$ are often referred to as boundary probability measures, as they concentrate on boundary surfaces.
To simplify the notation, we similarly use $\E_{\nu_i}[\cdot]$ (rather than $\E_{\nu_i^\uu}[\cdot]$) to denote expectation with respect to the boundary measures for $i=1,\ldots,d$, whenever the index $r$ is clear from the context.
Then, using the fact that $\mu^\uu = -R\delta^\uu$ in \eqref{eq1:necessary} and $\delta_i^\uu=r^i$ in Assumption \ref{assmpt: multiscale}, the BAR \eqref{eq:BAR87} can be rewritten as follows: for all $f\in C_b^2(\R_+^d)$:
\begin{equation}\label{eq:BAR}
   \frac{1}{2}\mathbb{E}_{\pi}\left[ \Linner{\Gamma, \nabla^2 f({Z^\uu}) } \right]  - \sum_{i=1}^{\Kdim} r^i \mathbb{E}_{\pi}\left[\Linner{R_{:,i},  \nabla f(Z^\uu)}\right]  +  \sum_{i=1}^{\Kdim} r^i \mathbb{E}_{\nu_i}\left[\Linner{R_{:,i},  \nabla f(Z^\uu)}\right]    =0.
\end{equation}

\subsection{Proof of $\caM$-matrix Case}\label{subsec: proof sketch}

In this section, we intend to illustrate the main ideas in the proof of Theorem \ref{thm: matrix P} by presenting the proof for the $\caM$-reflection matrix case in Corollary \ref{prop: M matrix}. This proof serves the start for the proof of the general case.

We denote by $\phi^\uu(\theta)\equiv \mathbb{E}_{\pi}[\exp(\theta^\T Z^\uu )]$ the moment generating function (MGF) of $Z^\uu$ under the stationary distribution $\pi^\uu$. Consequently, Theorem \ref{thm: matrix P} is equivalent to proving that: 
\begin{equation} \label{eq: convergence1}
	\lim_{r \to  0}\phi^\uu \left(r\eta_1, \ldots, r^\Kdim \eta_{\Kdim}\right)   =  \prod_{j=1}^d\frac{1}{1-m_j\eta_j}~,  \quad \text{ for all }\eta \in \R_-^\Kdim\equiv \{x\in \R^d:x\leq 0\},
\end{equation}
where $m_j$ is specified in \eqref{eq: mean k}. To obtain \eqref{eq: convergence1}, we employ mathematical induction to prove the following:
\begin{equation}\label{eq: convergence}
	\lim_{r \to  0}\phi^\uu(0,...,0,r^k\eta_k,r^{k+1}\eta_{k+1},  ...,r^\Kdim \eta_\Kdim)   =  \prod_{j=k}^d\frac{1}{1-m_j\eta_j}~,  \quad \text{ for all }\eta \in \R_-^\Kdim,
\end{equation}
progressing from $k=d+1$ back to $k=1$. The base case $k=d+1$ is trivially satisfied because $\phi^\uu(0,...,0)=1$ and $ \prod_{j=d+1}^d \cdot = 1$ by convention.  Assuming \eqref{eq: convergence} is true for $k+1$, we aim to prove its validity for $k$.  Specifically, we intend to show that, for any $\eta \in \R_-^\Kdim$:
\begin{equation}\label{eq: inductive equation}
	\lim_{r \to 0}\phi^\uu(0,...,0,r^k\eta_k,r^{k+1}\eta_{k+1},  ...,r^\Kdim \eta_\Kdim)=\frac{1}{1-m_k\eta_k}\lim_{r \to  0}\phi^\uu(0,...,0,0, r^{k+1}\eta_{k+1}, ...,r^\Kdim \eta_\Kdim).
\end{equation}
As \eqref{eq: inductive equation} automatically holds for $\eta_k=0$, we focus on proving the case where $\eta_k < 0$.

\paragraph{MGF-BAR.}
The key idea of our proof is to apply BAR \eqref{eq:BAR} to characterize the MGF. Specifically, by substituting exponential test functions $f_\theta(z)=\exp (\theta^\T z)$, where $\theta \in \R_{-}^\Kdim$, into BAR \eqref{eq:BAR}, we obtain the MGF-BAR for any $r\in (0,1)$:
\begin{equation}\label{eq:MGF-BAR}
	\frac{1}{2} \theta^\T \Gamma \theta \phi^\uu(\theta)+\sum_{i=1}^\Kdim r^i \theta^\T R_{:,i}  \left[\phi^\uu_i (\theta)-\phi^\uu (\theta)\right]=0, \quad  { \text{ for any }} \theta \in \mathbb{R}_{-}^d,
\end{equation}
where $\phi_i^\uu(\theta)\equiv \E_{\nu_i}[ f_\theta( Z^\uu ) ]$ denotes the MGF of $Z^\uu$ under the boundary measure $\nu_i^\uu$ for $i=1,\ldots,d$.

\paragraph{Test Functions.} To prove \eqref{eq: inductive equation} for each $k=1,\ldots,d$ and for any $\eta\in\R^d_{-}$ with $\eta_k<0$, we will utilize the MGF-BAR in \eqref{eq:MGF-BAR} with two test functions $f_{\theta^\uu_k},f_{\tilde{\theta}^\uu_k}$ for all $r\in (0,1)$. The vectors $\theta^\uu_k$ and $\tilde{\theta}^\uu_k$ are defined as:
\begin{equation}\label{eq: theta}
	\theta_k^{(r)} = \sum_{\ell=k}^d r^\ell\eta_\ell\cdot u_\ell ,\quad \tilde{\theta}_k^{(r)}= r^{k+1/2}\eta_k \cdot u_k +\sum_{\ell=k+1}^d r^\ell\eta_\ell\cdot u_\ell,
\end{equation}
where the vectors $\{u_\ell\}_{\ell=1}^d$ are specified by \eqref{eq: u}. The vectors $\theta_k^{(r)}$ and $\tilde{\theta}_k^{(r)}$ are designed to satisfy
\begin{equation} \label{eq: theta property}
	(\theta_k^\uu)^\T R_{:,i}=0 \text{ for }i<k, \quad (\tilde\theta_k^\uu)^\T R_{:,i}=0 \text{ for }i<k,
\end{equation}
since $u_k^\T R_{:,i}=0$ for $i<k$ by \eqref{eq: w}. Under the assumption that the reflection matrix $R$ is an $\caM$-matrix, we further have $\theta_k^{(r)}, \tilde\theta_k^{(r)} \in \R^d_-$ by Remark \ref{rmk: w}.
Hence, $f_{\theta_k^{(r)}}$ and $f_{\tilde \theta_k^{(r)}}$ are both in $C_b^2(\R_+^d)$.

Substituting $f_{\theta_k^{(r)}}$ into the MGF-BAR \eqref{eq:MGF-BAR} and utilizing \eqref{eq: theta property}, 
we obtain
\begin{equation*}
	\frac{1}{2} (\theta_k^{(r)})^\T \Gamma \theta_k^{(r)} \phi^\uu(\theta_k^{(r)})+\sum_{i=k}^\Kdim r^i (\theta_k^{(r)})^\T R_{:,i}  \left[\phi^\uu_i (\theta_k^{(r)})-\phi^\uu (\theta_k^{(r)})\right]=0.
\end{equation*}
By the design of $\theta_k^{(r)}$ in \eqref{eq: theta}, we have $\theta_k^{(r)} = r^k\eta_k\cdot u_k +o(r^{k})$. Therefore, we get
\begin{align*}
	&\frac{1}{2} \left(r^k\eta_k\cdot u_k^\T +o(r^{k})\right) \Gamma \left( r^k\eta_k\cdot u_k +o(r^{k}) \right) \phi^\uu(\theta_k^{(r)})\\
	&\quad +\sum_{i=k}^\Kdim r^i \left(r^k\eta_k\cdot u_k^\T +o(r^{k})\right) R_{:,i}  \left[\phi^\uu_i (\theta_k^{(r)})-\phi^\uu (\theta_k^{(r)})\right]=0,
\end{align*}
which simplifies to
\begin{equation}  \label{eq: derivation11}
	\frac{1}{2}r^{2k}\eta_k^2 \cdot u_k'\Gamma u_k\phi^{(r)}(\theta^{(r)}_k) + r^{2k}\eta_k \cdot u_k'R_{:,k} \left[\phi_k^{(r)}(\theta^{(r)}_k)-\phi^{(r)}(\theta^{(r)}_k)\right]+o(r^{2k})=0.
\end{equation}
Since $R$ is an $\caM$-matrix, we have the following result concerning the second term in \eqref{eq: derivation11}. The proof of Lemma \ref{lmm: uk} is provided in Appendix \ref{subsec: proof lemmas}.
\begin{lemma}  \label{lmm: uk}
	When $R$ is a $\caP$-matrix, for all $k=1,...,d$, we have
	\begin{equation*}
		u_k'R_{:,k}>0.
	\end{equation*}  
\end{lemma}
Consequently, we conclude that the second term is also of order $r^{2k}$. In addition, as $\eta_k<0$, we have
\begin{equation} \label{eq: derivation1}
	\phi^{(r)}(\theta^{(r)}_k) = \frac{1}{1-m_k\eta_k}\phi_k^{(r)}(\theta^{(r)}_k)+o(1),\quad  \text{ with }m_k=\frac{u_k'\Gamma u_k}{2u_k'R_{:,k}},
\end{equation}
which establishes the connection between $\phi^{(r)}(\theta^{(r)}_k)$ and $\phi_k^{(r)}(\theta^{(r)}_k)$. Since the covariance matrix $\Gamma$ is positive definite, we have  $m_k>0$.\\

To construct the relation between $\phi_k^{(r)}(\theta^{(r)}_k)$ and $\phi^{(r)}(\theta^{(r)}_{k})$, we substitute $f_{\tilde{\theta}_k^{(r)}}$ into the MGF-BAR \eqref{eq:MGF-BAR}, and utilizing \eqref{eq: theta property}, we obtain
\begin{equation*}
	\frac{1}{2} (\tilde\theta_k^{(r)})^\T \Gamma \tilde\theta_k^{(r)} \phi^\uu(\tilde\theta_k^{(r)})+\sum_{i=k}^\Kdim r^i (\tilde\theta_k^{(r)})^\T R_{:,i}  \left[\phi^\uu_i (\tilde\theta_k^{(r)})-\phi^\uu (\tilde\theta_k^{(r)})\right]=0.
\end{equation*}
From the definition in \eqref{eq: theta}, we have $\tilde\theta_k^{(r)} = r^{k+1/2}\eta_k\cdot u_k +o(r^{k+1/2})$, and hence, we get
\begin{align*}
	&\frac{1}{2} \left(r^{k+1/2}\eta_k\cdot u_k^\T +o(r^{{k+1/2}})\right) \Gamma \left( r^{k+1/2}\eta_k\cdot u_k +o(r^{{k+1/2}}) \right) \phi^\uu(\tilde\theta_k^{(r)})\\
	&\quad +\sum_{i=k}^\Kdim r^i \left(r^{k+1/2}\eta_k\cdot u_k^\T +o(r^{{k+1/2}})\right) R_{:,i}  \left[\phi^\uu_i (\tilde\theta_k^{(r)})-\phi^\uu (\tilde\theta_k^{(r)})\right]=0.
\end{align*}
Therefore, we have
\begin{equation} \label{eq: derivation22}
	r^{2k+1/2}\eta_k\cdot u_k^\T  R_{:,k}  \left[\phi^\uu_k (\tilde\theta_k^{(r)})-\phi^\uu (\tilde\theta_k^{(r)})\right]+o(r^{{2k+1/2}})=0.
\end{equation}
Since $\eta_k<0$ and $u_k'R_{:,k}>0$ as shown in Lemma \ref{lmm: uk}, we have
\begin{equation} \label{eq: derivation2}
	\phi_k^\uu(\tilde\theta^{(r)}_k)-\phi^\uu(\tilde\theta^{(r)}_k)+o(1) = 0.
\end{equation}

To derive \eqref{eq: inductive equation} from \eqref{eq: derivation1} and \eqref{eq: derivation2}, we need to establish the following lemma, which depends on the uniform moment bound condition in \eqref{eq: moment bound}. Intuitively, for each $j=1,...,d$, \eqref{eq: moment bound} shows that $r^jZ^\uu_j$ is uniformly bounded in expectation as $r\to 0$. Consequently, $r^{j+\epsilon}Z^\uu_j$ converges to zero in probability as $r\to 0$ for any $\epsilon>0$. Therefore, the $j$th components of $\theta^{(r)}_k$ and $\tilde{\theta}^{(r)}_k$ are negligible if their order is $o(r^j)$.
The detailed proof of Lemma \ref{lmm:SSCs} is provided in Appendix \ref{subsec: proof lemmas}.

\begin{lemma}\label{lmm:SSCs} When $R$ is an $\caM$-matrix and Assumption \ref{assmpt: multiscale} holds, for all $k=1,...,d$,  as $r\to 0$, we have  
	\begin{align*} 
		&\phi^\uu(\theta^{(r)}_k)-\phi^\uu(0,...,0,r^k\eta_k,  r^{k+1}\eta_{k+1},...,r^\Kdim \eta_\Kdim)= o(1),\\
		&\phi^\uu_k(\theta^{(r)}_k)-\phi^\uu_k(0,...,0,~~~~0,r^{k+1}\eta_{k+1},  ...,r^\Kdim \eta_\Kdim)= o(1),\\
		&\phi^\uu(\tilde{\theta}^{(r)}_k)-\phi^\uu(0,...,0,~~~~0,r^{k+1}\eta_{k+1},  ...,r^\Kdim \eta_\Kdim)= o(1), \\
		&\phi^\uu_k(\tilde{\theta}^{(r)}_k)-\phi^\uu_k(0,...,0,~~~~0,r^{k+1}\eta_{k+1},  ...,r^\Kdim \eta_\Kdim)= o(1).
	\end{align*}
\end{lemma}

\paragraph{Proof of Corollary \ref{prop: M matrix}.} Lemma \ref{lmm:SSCs} implies that \eqref{eq: derivation1} becomes
\begin{equation}\label{eq: BAR with theta}
	\phi^{(r)}(0,...,0,r^k\eta_k,  r^{k+1}\eta_{k+1},...,r^\Kdim \eta_\Kdim) = \frac{1}{1-m_k\eta_k}\phi_k^{(r)}(0,...,0,0,  r^{k+1}\eta_{k+1},...,r^\Kdim \eta_\Kdim)+o(1),
\end{equation}
and \eqref{eq: derivation2} becomes
\begin{equation}\label{eq: phi_k and phi}
\phi^{(r)}_k(0,...,0,0,r^{k+1}\eta_{k+1},...,r^d\eta_d)=\phi^{(r)}(0,...,0,0,r^{k+1}\eta_{k+1},  ...,r^d\eta_d)+ o(1) .
\end{equation}
Combining \eqref{eq: BAR with theta} and \eqref{eq: phi_k and phi}, we can conclude
$$\phi^{(r)}(0,...,0,r^k\eta_k,...,r^d\eta_d)=\frac{1}{1-m_k\eta_k}\phi^{(r)}(0,...,0,0,r^{k+1}\eta_{k+1},  ...,r^d\eta_d)+o(1),$$
and thus complete the proof of \eqref{eq: inductive equation} and Corollary \ref{prop: M matrix}.\\

\subsection{Proof of Non-$\caM$-matrix Case}\label{subsec: proof sketch P}

The proof of Theorem \ref{thm: matrix P} is based on the same main ideas in our analysis for the case with an $\caM$-reflection matrix in Section \ref{subsec: proof sketch}. However, the key technical issue arises from the fact that when the reflection matrix $R$ is a $\caP$-matrix, the vectors $\theta^\uu_k$ and $\tilde\theta^\uu_k$ in \eqref{eq: theta} may contain negative elements. As a result, the test functions $f_{\theta^\uu_k}$ and $f_{\tilde\theta^\uu_k}$  are not guaranteed to belong to $C_b^2(\R_+^d)$, which means they cannot be directly applied to the MGF-BAR \eqref{eq:MGF-BAR}. To address this issue, we design truncated versions of these test functions. The truncation will introduce additional terms into the resulting BAR equations. However, we show that these extra terms do not affect the asymptotic convergence, allowing us to extend the proof ideas to the case with $\caP$-reflection matrices.

\paragraph{Truncated Test Functions.}
To make $\theta^\uu_k$ and $\tilde\theta^\uu_k$ in the test functions applicable for the case of a  $\caP$-reflection matrix, we introduce the truncated exponential test functions as follows: for $\theta\in \R^d$,
\begin{equation*}  % \label{eq:test}
    \tilde{f}_\theta(z) = \exp\left( \sum_{i=1}^d \kappa(\theta_iz_i) \right) \in C_b^2(\R^d_+),
\end{equation*}
where $\kappa$ is an upper-truncated function defined as
\begin{equation*}
    \kappa(x) = \begin{cases}
        x & x\leq 1\\
        3x^5 - 22x^4 + 62x^3 - 84x^2 + 56x - 14& 1 < x\leq 2\\
        2 & x>2.
    \end{cases},
\end{equation*}

To ensure that $\tilde{f}_\theta \in C_b^2(\R^d_+)$, we need $\kappa$ to be $C^2(\R)$ and upper bounded. Hence, we design $\kappa$ as an identity mapping for $x \leq a$ and a constant for $x \geq b$ (where $b > a > 0$), with polynomial interpolation in between. For the function to be $C^2(\R)$, $\kappa$ must have continuous function values, first derivatives, and second derivatives at the boundary points. This requires the interpolating polynomial to be of at least degree 5. In our specific case, we set $a=1$ and $b=2$, and the function $\kappa$, its first derivative $\dot{\kappa}$ and second-order derivative $\ddot{\kappa}$ are illustrated in Figure~\ref{fig:kappa} below.
\begin{figure}[htbp]
  \centering
  \includegraphics[width=1.\textwidth]{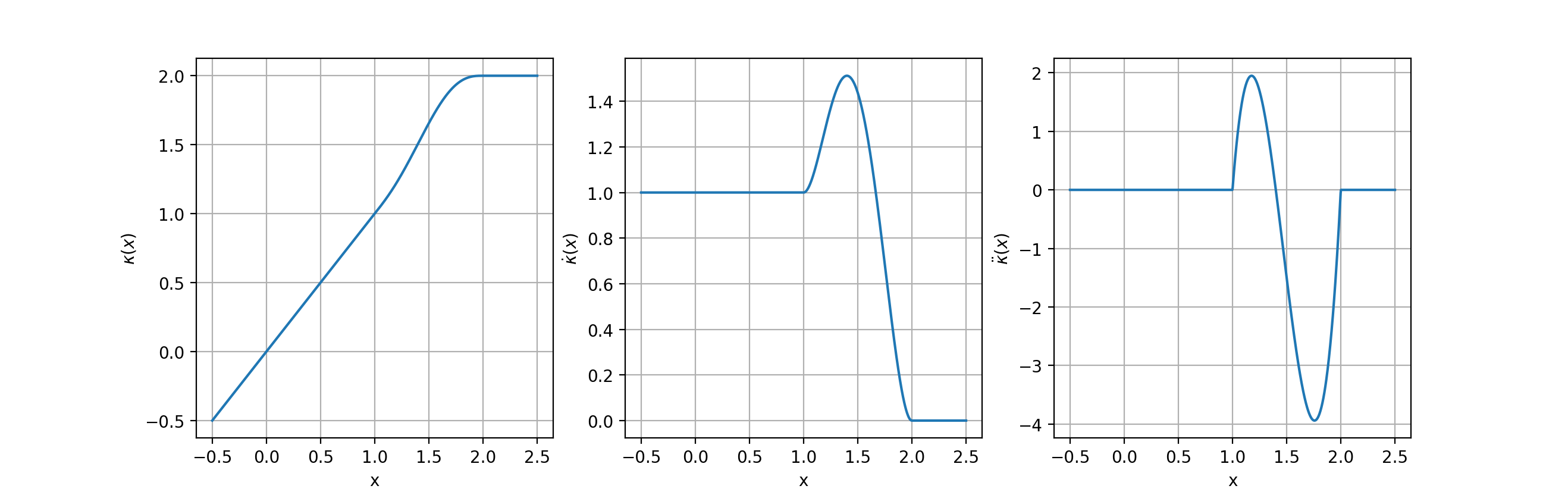}
  \caption{The function $\kappa$ and its first and second derivatives for $x\in [-0.5,2.5]$.}
  \label{fig:kappa}
\end{figure}

The function $\kappa$ has the following properties:
\begin{equation}\label{eq:kappa properties}
	\kappa(x)\leq 2, \abs{\dot{\kappa}(x)}\leq 2, \abs{\ddot{\kappa}(x)}\leq 4\I{1<x<2}, \kappa(x)-x\leq \I{x>1}, \abs{\dot{\kappa}(x)-1}\leq \I{x>1}.
\end{equation}
Similar to MGFs $\phi^\uu$, we define $\psi^\uu(\theta)\equiv \mathbb{E}_{\pi}[\tilde{f}_\theta(Z ^\uu)]$ and $\psi_i^\uu(\theta)\equiv \E_{\nu_i}[ \tilde f_\theta( Z  ^\uu) ]$ for $i=1,\ldots,d$. When $\theta\in \R_-^d$, we have $\tilde{f}_\theta(z)={f}_\theta(z)$ for any $z\in \R_+^d$, and hence, $\psi^\uu(\theta)=\phi^\uu(\theta)$ and $\psi_i^\uu(\theta)=\phi^\uu_i(\theta)$.  

Substituting the test functions $\tilde{f}_\theta(z)$ into the BAR \eqref{eq:BAR}, we obtain the truncated MGF-BAR, which is similar to MGF-BAR \eqref{eq:MGF-BAR} but includes an extra term:
\begin{equation}\label{eq:TMGF-BAR}
    \begin{aligned}
        &\frac{1}{2} \theta^\T\Gamma\theta \psi^\uu (\theta) +\sum_{i=1}^\Kdim r^i \theta^\T R_{:,i} \left[\psi_i^\uu (\theta) -\psi^\uu (\theta) \right]+\gamma^\uu(\theta)=0,
    \end{aligned}
\end{equation}
where the extra term $\gamma(\theta)$ is given by
\begin{equation} \label{eq: gamma}
	\begin{aligned}
		\gamma^\uu(\theta) &\equiv \frac{1}{2} \E_{\pi}\left[ \Linner{\Gamma, \epsilon_{H,\theta}\left( Z^\uu \right) }  \tilde{f}_\theta\left( Z^\uu \right)\right] \\
		&\quad  +\sum_{i=1}^\Kdim r^i \left(\E_{\nu_i}\left[  \epsilon_{g, \theta}^\T\left( Z^\uu \right) R_{:,i}  \tilde{f}_\theta\left( Z^\uu \right)\right] -\E_{\pi}\left[  \epsilon_{g, \theta}^\T\left( Z^\uu \right) R_{:,i}   \tilde{f}_\theta\left( Z^\uu \right) \right] \right).
	\end{aligned}
\end{equation}
The extra term corresponding to the gradient is given by $\epsilon_{g, \theta, j}(z)  = \theta_j [ \dot{\kappa}(\theta_j z_j) - 1]$ for $j=1,\ldots,d$, and the extra term corresponding to the Hessian is given by $\epsilon_{H, \theta, ij}(z)  = \theta_i\theta_j [\dot{\kappa}(\theta_i z_i)\dot{\kappa}(\theta_j z_j) - 1]+\theta_i^2\ddot{\kappa}(\theta_i z_i)\I{i=j}$ for $i,j\in \{1,\ldots,d\}$. 

We observe that for all $z\geq 0$,
\begin{equation} \label{eq: eps0}
	\epsilon_{g, \theta, j}(z) = 0, \quad \epsilon_{H, \theta, ij}(z) = 0, \quad \text{for } \theta_{i}, \theta_j\leq0.
\end{equation}
Therefore, when $\theta\in \R_-^d$, the extra term $\gamma^\uu(\theta)$ in \eqref{eq: gamma} is zero, and the truncated MGF-BAR \eqref{eq:TMGF-BAR} degenerates to the MGF-BAR \eqref{eq:MGF-BAR}.

\paragraph{Technical Lemmas.}

Since the test functions involve the truncation function $\kappa$, we need to introduce two technical lemmas.
Lemma \ref{lmm:SSCs P} is similar to Lemma \ref{lmm:SSCs}, with the only difference being the substitution of $\phi$ with $\psi$. 
Lemma \ref{lmm:extra} demonstrates that the extra terms \eqref{eq: gamma} in the truncated MGF-BAR \eqref{eq:TMGF-BAR} are negligible. Condition \ref{assmpt: moment bound P} provides the foundation for Lemmas \ref{lmm:SSCs P} and \ref{lmm:extra}. The proofs of Lemmas \ref{lmm:SSCs P} and \ref{lmm:extra} are provided in Appendix~\ref{subsec: proof lemmas}.

\begin{lemma}\label{lmm:SSCs P} When Assumptions \ref{assmpt: multiscale}-\ref{assmpt: reflect matrix} and Condition \ref{assmpt: moment bound P} hold, for all $k=1,...,d$,  as $r\to 0$, we have  
	\begin{align} 
		&\psi^\uu(\theta^{(r)}_k)-\phi^\uu(0,...,0,r^k\eta_k,  r^{k+1}\eta_{k+1},...,r^\Kdim \eta_\Kdim)= o(1), \label{eq: SSC A P}\\
		&\psi^\uu_k(\theta^{(r)}_k)-\phi^\uu_k(0,...,0,~~~~0,r^{k+1}\eta_{k+1},  ...,r^\Kdim \eta_\Kdim)= o(1), \label{eq: SSC B P}\\
		&\psi^\uu(\tilde{\theta}^{(r)}_k)-\phi^\uu(0,...,0,~~~~0,r^{k+1}\eta_{k+1},  ...,r^\Kdim \eta_\Kdim)= o(1), \label{eq: SSC C P}\\
		&\psi^\uu_k(\tilde{\theta}^{(r)}_k)-\phi^\uu_k(0,...,0,~~~~0,r^{k+1}\eta_{k+1},  ...,r^\Kdim \eta_\Kdim)= o(1). \label{eq: SSC D P}
	\end{align}
\end{lemma}

\begin{lemma}[Properties of Extra Terms]\label{lmm:extra}  
	When Assumptions \ref{assmpt: multiscale}-\ref{assmpt: reflect matrix} and Condition \ref{assmpt: moment bound P} hold, as $r\to 0$, we have
	\begin{equation*}
		\gamma^\uu(\theta^{(r)}_k) = o(r^{2k}), \quad \gamma^\uu(\tilde{\theta}^{(r)}_k) = o(r^{2k+1/2}),
	\end{equation*}
	where $\gamma^\uu$ is defined in \eqref{eq: gamma}.
\end{lemma}

When $\theta^\uu_k, \tilde{\theta}^\uu_k \in \R_-^d$, Lemma \ref{lmm:SSCs P} will reduce to Lemma \ref{lmm:SSCs}, and Lemma \ref{lmm:extra} will be trivially satisfied since $\gamma^\uu(\theta^{(r)}_k)=\gamma^\uu(\tilde{\theta}^{(r)}_k)=0$. 

\paragraph{Proof of Theorem \ref{thm: matrix P}.} 
The proof of Theorem \ref{thm: matrix P} is similar to that of Corollary \ref{prop: M matrix} in Section \ref{subsec: proof sketch}. Specifically, by applying $\theta^{(r)}_k$ into the truncated MGF-BAR \eqref{eq:TMGF-BAR}, the properties of $\theta^{(r)}_k$ in \eqref{eq: theta property} and Lemma \ref{lmm:extra} imply that: 
\begin{align*}
	&\frac{1}{2}r^{2k}\eta_k^2 \cdot u_k'\Gamma u_k\psi^{(r)}(\theta^{(r)}_k) + r^{2k}\eta_k \cdot u_k'R_{:,k} \left[\psi_k^{(r)}(\theta^{(r)}_k)-\psi^{(r)}(\theta^{(r)}_k)\right]+o(r^{2k})=0,
  \end{align*}
which is similar to \eqref{eq: derivation11}, except that the extra term $\gamma^\uu(\theta^{(r)}_k)$ appears with order $o(r^{2k})$ by Lemma \ref{lmm:extra}. According to Lemma \ref{lmm: uk}, the second term is also of order $r^{2k}$. Hence, Lemma~\ref{lmm:SSCs P} imply the same result as \eqref{eq: BAR with theta}. Since the covariance matrix $\Gamma$ is positive definite, we have $m_k > 0$.

By applying $\tilde{\theta}^{(r)}_k$ into the truncated MGF-BAR \eqref{eq:TMGF-BAR}, the properties of $\tilde{\theta}^{(r)}_k$ in \eqref{eq: theta property} and Lemma \ref{lmm:extra} imply that
\begin{equation*}
	r^{2k+1/2} \eta_k \cdot u_k^\T R_{:,k}  \left[\psi_k^\uu(\tilde\theta^{(r)}_k)-\psi^\uu(\tilde\theta^{(r)}_k)\right]+o(r^{2k+1/2})  = 0,
\end{equation*}
which is similar to \eqref{eq: derivation22}, except that the extra term $\gamma^\uu(\tilde{\theta}^{(r)}_k)$ appears with order $o(r^{2k+1/2})$ by Lemma \ref{lmm:extra}.  Hence, Lemma \ref{lmm:SSCs P} imply the same result of \eqref{eq: phi_k and phi}.

Therefore, \eqref{eq: BAR with theta} and \eqref{eq: phi_k and phi} together establish the proof of \eqref{eq: inductive equation} and complete the proof of Theorem \ref{thm: matrix P}.

\section{Uniform Moment Bounds} \label{sec: uniform moment bound}

In Sections \ref{subsec: uniform moment bound}-\ref{subsec: FBFS P matrix moment}, we prove Condition \ref{assmpt: moment bound P}, for $\caM$-reflection matrix, $2$-dimensional case and lower triangular $\caP$-reflection matrix, respectively. Such uniform moment bounds play a pivotal role in proving Corollary \ref{prop: M matrix}-\ref{prop: FBFS reentrant Queue}, respectively. To simplify the calculation and better illustrate the main proof ideas, we formally apply unbounded test functions in the proofs throughout Sections \ref{subsec: uniform moment bound}-\ref{subsec: FBFS P matrix moment}. In Section \ref{subsec: truncate}, we make the proofs rigorous, explaining how to replace the unbounded test functions with the truncated ones in the proofs.

By using the definition of boundary distribution \eqref{eq: nu}, we can restate \eqref{eq: moment bound} as follows: for any $i,j\in \{1,...,d\}$,
\begin{align}  \label{eq: moment bound P}
	\E_{\pi} \left[ \left( r^j Z_j^\uu \right)^d \right] =O(1), \quad 
	\E_{\nu_i}\left[ \left( r^{j}Z_j^\uu \right)^d \right] = o(1/r) \quad \text{as } r \to 0.
\end{align}

\subsection{\texorpdfstring{$\caM$}{M}-matrix Case}\label{subsec: uniform moment bound}

To prove \eqref{eq: moment bound P}, it is sufficient to prove that for any $k=1,\ldots, d$ and any $n\in \Z_+$, there exist positive constants $C_{k,n}, C_{1, k,n}, \ldots, C_{d, k,n}<\infty$ such that for all $r\in (0,r_0)$,
\begin{align}
	&\E_{\pi} \left[\left(r^k Z_k^\uu\right)^{n}\right] < C_{k,n}, \quad r^{i-1}\E_{\nu_i}\left[ \left( r^{k}Z_k^\uu \right)^n \right] < C_{i,k,n} , \quad  \forall i = 1,\ldots,d,  \label{eq: moment bound induction} 
\end{align}
where $r_0\in (0,1]$ is a constant independent of $k$ and $n$, and will be determined later in Lemma~\ref{lmm: eta}. 
The first statement of \eqref{eq: moment bound P} is a direct corollary of the first one in \eqref{eq: moment bound induction} for $n=1$. 
The second statement of \eqref{eq: moment bound P} can be proved by Hölder's inequality as follows: for any $r\in (0,r_0)$ and any $i,k\in\{1,\ldots, d\}$,
\begin{equation*}
r\E_{\nu_i} \left[ r^k Z_k^\uu \right] \leq \left( r^i\E_{\nu_i} \left[ \left( r^k Z_k^\uu \right)^{i} \right] \right)^{1/i} \leq \left( r C_{i,k,i} \right)^{1/i} = o(1), \text{ as $r\to 0$.}
\end{equation*}

\begin{lemma}\label{lmm: eta}
Define
\begin{equation*}
	r_0 \equiv  \min_{\{1\leq k\leq i \leq d \mid u_k^\T R_{:,i} \neq 0\}}   \left( \frac{u_k^\T R_{:,k}}{\Kdim \abs{u_k^\T R_{:,i}}} \right)^{\frac{1}{i-k}}   .
\end{equation*}
Then, $r_0>0$ and for any $r\in (0,r_0)$, and any $k = 1, \ldots, \Kdim$, we have
	\begin{equation} \label{eq: r0 condition}
		\sum_{i=k}^\Kdim r^{i-k} u_k^\T R_{:,i} = u_k^\T R_{:,k} + \sum_{i=k+1}^\Kdim r^{i-k} u_k^\T R_{:,i}  \geq \frac{1}{\Kdim }    u_k^\T R_{:,k}.
	\end{equation}
\end{lemma}
\begin{proof}[Proof of Lemma \ref{lmm: eta}.] 
	According to Lemma \ref{lmm: uk} that $u_k^\T R_{:,k}>0$,
	$r < r_0$ implies that
    \begin{equation*}
        r < \left( \frac{u_k^\T R_{:,k}}{\Kdim \abs{u_k^\T R_{:,i}}} \right)^{\frac{1}{i-k}}  \quad \text{ for any $1\leq k \leq i \leq d$ and $u_k^\T R_{:,i} \neq 0$}.
    \end{equation*}
    Raising both sides of the inequality to the power of $i-k$ and then multiplying both sides by $|u_k^\T R_{:,i}|$, we have
    \begin{equation*}
        \abs{u_k^\T R_{:,i}}r^{i-k} <    \frac{1}{\Kdim }u_k^\T R_{:,k}    \quad \text{ for any $1\leq k \leq i \leq d$ and $u_k^\T R_{:,i} \neq 0$}.
    \end{equation*}
    Therefore, for any $k = 1, \ldots, d$,
    \begin{equation*}
		\sum_{i=k+1}^\Kdim r^{i-k} u_k^\T R_{:,i}  \leq   \sum_{i=k+1}^\Kdim r^{i-k} \abs{u_k^\T R_{:,i}} \leq  \frac{d-k}{d }u_k^\T R_{:,k}   \leq u_k^\T R_{:,k}  - \frac{1}{d }u_k^\T R_{:,k} ,
    \end{equation*}
    which implies \eqref{eq: r0 condition}.
\end{proof}

In the following part, we will use the mathematical induction with two directions to prove \eqref{eq: moment bound induction}. Here, unlike the induction argument for our main result in \eqref{eq: inductive equation}, we shall start from the least congested component 1 to the most congested component $d$. Specifically, we will prove \eqref{eq: moment bound induction} for a fixed $k$ by assuming that they hold for any $\ell=1,\ldots,k-1$ ($k=1$ for empty information). However, when proving \eqref{eq: moment bound induction} for a given $k$, we will prove them for a fixed $n$ by assuming that they hold for any $m=0,1,\ldots,n-1$. 

Suppose that \eqref{eq: moment bound induction} holds for any $\ell=1,\ldots,k-1$, and we will proceed to prove them for $k$. Note that the base case for $k$ holds, as \eqref{eq: moment bound induction} is naturally satisfied for $n=0$. We further assume that \eqref{eq: moment bound induction} hold for any $m=0,1,\ldots,n-1$, and we will prove them for $n$.
	
To prove the first statement in \eqref{eq: moment bound induction}, we apply the test function as follows into BAR \eqref{eq:BAR} 
\begin{equation} \label{eq: test function moment}
f_{k,n}\left( z \right)=\frac{r^{-2k}}{n+1}\left( r^k u_k^\T z\right)^{n+1},
\end{equation}
where $u_k$ is defined in \eqref{eq: u}. To shorten the notation, we will write $y_k$ for $r^k u_k^\T z$, and then $\nabla_z f_{k,n}(z) = r^{-k}y_k^nu_k$ and $\nabla^2_z f_{k,n}(z) = ny_k^{n-1}u_ku_k^\T$. According to Remark \ref{rmk: w} that $u_k\in \R_+^d$, we also have the relation 
\begin{equation} \label{eq: yk}
0\leq (r^k z_k)^n \leq y_k^n  \leq k^n   (r^k z_{\max})^n \leq k^n   (r^k z_k)^n +k^n  \sum_{j=1}^{k-1} (r^k w_{jk} z_j)^n ,
\end{equation}
where $z_{\max}\equiv \max(z_k, \max_{j<k} w_{jk} z_j )$.
Then, BAR \eqref{eq:BAR} implies
\begin{equation} \label{eq: BAR moment bound}
\frac{n}{2} u_k^\T\Gamma  u_k \mathbb{E}_{\pi}\left[ \left( Y^\uu_k \right)^{n-1}  \right] - \sum_{i=1}^{\Kdim} r^{i-k} u_k^\T R_{:,i} \mathbb{E}_{\pi}\left[  \left( Y^\uu_k \right)^{n}\right]    +   \sum_{i=1}^{\Kdim} r^{i-k}  u_k^\T R_{:,i}\mathbb{E}_{\nu_i}\left[   \left(  Y^\uu_k \right)^{n}\right]     =0.
\end{equation}
For any $i > k$, since the off-diagonal elements of $R$ are nonpositive and $u_k\in \R_+^d$, we have
\begin{equation} \label{eq: uk neg}
u_k^\T R_{:,i} = [w_{1k}, \ldots, w_{k-1,k}] R_{[1:k-1], i} +  R_{ki} = \sum_{j=1}^{k-1} w_{jk}R_{ji} +  R_{ki} \leq 0.
\end{equation}
Together with $u_k'R_{:,i}=0$ for $i<k$ by \eqref{eq: w}, hence, we have
\begin{align}
&\sum_{i=k}^{\Kdim} r^{i-k} u_k^\T R_{:,i}  \mathbb{E}_{\pi}\left[  \left(Y^\uu_k\right)^{n}\right] =  \frac{n}{2} u_k^\T\Gamma  u_k \mathbb{E}_{\pi}\left[ \left( Y^\uu_k \right)^{n-1}  \right]    +   \sum_{i=k}^{\Kdim} r^{i-k}  u_k^\T R_{:,i}\mathbb{E}_{\nu_i}\left[   \left(  Y^\uu_k \right)^{n}\right]  \notag \\
&\quad  \leq \frac{n}{2} u_k^\T\Gamma  u_k \mathbb{E}_{\pi}\left[ \left( Y^\uu_k \right)^{n-1}  \right]    +    u_k^\T R_{:,k}\mathbb{E}_{\nu_k}\left[   \left( Y_k^\uu\right)^{n}\right] \notag \\
&\quad  \leq \frac{n}{2} u_k^\T\Gamma  u_k k^{n-1} \mathbb{E}_{\pi}\left[   \left(r^k Z^\uu_k\right)^{n-1} +  \sum_{j=1}^{k-1} \left(r^k w_{jk} Z^\uu_j\right)^{n-1}   \right]    +    u_k^\T R_{:,k}k^n\mathbb{E}_{\nu_k}\left[  \sum_{j=1}^{k-1} \left(r^k w_{jk} Z^\uu_j\right)^{n}  \right] \notag \\
&\quad \leq \frac{n}{2} u_k^\T\Gamma  u_k k^{n-1}  \left( C_{k,n-1} +  \sum_{j=1}^{k-1} w_{jk}^{n-1} C_{j,n-1} \right)    +    u_k^\T R_{:,k}k^n \sum_{j=1}^{k-1} w_{jk}^{n} C_{k,j,n}, \label{eq: moment computation}
\end{align}
where the second inequality follows from the relation \eqref{eq: yk} and that $Z_k=0$ almost surely under the boundary measure $\nu_k$, and the last inequality holds due to the induction hypotheses.
According to Lemma \ref{lmm: eta}, for any $r\in (0,r_0)$, we have 
\begin{equation}
\mathbb{E}_{\pi}\left[  \left(r^{k} Z^\uu_k \right)^{n}\right]\leq \mathbb{E}_{\pi}\left[  \left(Y^\uu_k \right)^{n}\right] \leq  \frac{ n\Kdim u_k^\T\Gamma  u_k}{2 u_k^\T R_{:,k}}  k^{n-1}  \left( C_{k,n-1} +  \sum_{j=1}^{k-1} w_{jk}^{n-1} C_{j,n-1} \right)    + dk^n \sum_{j=1}^{k-1} w_{jk}^{n} C_{k,j,n}. \label{eq: moment bound res}
\end{equation}
Thus, the first statement of \eqref{eq: moment bound induction} holds for $k$ and $n$ with $C_{k,n}$ defined above. 

Since $Z_k=0$ almost surely under the boundary measure $\nu_k$, we can obtain that $r^{k-1}\E_{\nu_k}[ ( r^{k}Z_k^\uu )^n ]=0$ and hence, $C_{k,n,k}=0$.
To prove the second statement of \eqref{eq: moment bound induction} under other boundary measures, we apply the test function as follows into BAR \eqref{eq:BAR}
\begin{equation} \label{eq: test func nu}
f_{j,k,n}\left( z \right)=r^{-k-1}\left( r^k z_k \right)^n\exp\left( -r^k z_j \right)\qquad  j\in \{1,2,\ldots, \Kdim\}\setminus \{k\}.
\end{equation}
Substituting $f_{j,k,n}$ into BAR \eqref{eq:BAR}, we have
\begin{align}
&\frac{r^{k-1}}{2}\mathbb{E}_{\pi}\left[  \left( \Gamma_{kk}n(n-1) \left( r^k Z^\uu_k \right)^{n-2} - 2 \Gamma_{kj} n  \left( r^k Z^\uu_k \right)^{n-1} + \Gamma_{jj}  \left( r^k Z^\uu_k \right)^{n} \right) \exp\left( -r^k Z^\uu_j \right) \right]  \notag \\
&\quad  - \sum_{i=1}^{\Kdim} r^{i-1} \mathbb{E}_{\pi}\left[ \left( R_{ki} n  \left( r^k Z^\uu_k \right)^{n-1}  -R_{ji}\left( r^k Z^\uu_k \right)^n \right)\exp\left( -r^k Z^\uu_j \right)\right] \notag \\
&\quad  +  \sum_{i=1,i\neq k}^{\Kdim} r^{i-1}  \mathbb{E}_{\nu_{i}}\left[\left( R_{ki} n  \left( r^k Z^\uu_k \right)^{n-1}  -R_{ji}\left( r^k Z^\uu_k \right)^n \right)\exp\left( -r^k Z^\uu_j \right)\right]  =0 \label{eq: BAR moment nu}.
\end{align}
Hence, moving the expectation under the boundary measure $\nu_j$ to the one side and omitting the negative parts and the exponential decay on the other side, we have
\begin{align}
& r^{j-1} R_{jj} \E_{\nu_j}\left[ \left(r^k  Z^\uu_k\right)^n\right]   \leq \frac{r^{k-1}}{2}\mathbb{E}_{\pi}\left[   \Gamma_{kk}n(n-1) \left( r^k Z^\uu_k \right)^{n-2} + 2\abs{\Gamma_{kj}}n \left( r^k Z^\uu_k \right)^{n-1}   + \Gamma_{jj}  \left( r^k Z^\uu_k \right)^{n}    \right] \notag \\
&\quad  + \sum_{i=1,i\neq k}^{\Kdim} r^{i-1} (- R_{ki} )n  \mathbb{E}_{\pi}\left[  \left( r^k Z^\uu_k\right)^{n-1}   \right]  +   r^{j-1} R_{jj}  \mathbb{E}_{\pi}\left[ \left( r^k Z^\uu_k \right)^n  \right]    +  \sum_{i=1,i\neq k,j}^{\Kdim} r^{i-1}(-R_{ji})  \mathbb{E}_{\nu_{i}}\left[ \left( r^k Z^\uu_k \right)^n   \right]  \label{eq: moment computation nu}.
\end{align}
Moving the expectation under boundary measures on the same side and applying the induction hypothesis, we have, for any $r\in (0,r_0)$,
\begin{align} \label{eq: cj}
\sum_{i=1,i\neq k }^{\Kdim} r^{i-1}  R_{ji}  \mathbb{E}_{\nu_i}\left[ \left( r^k Z^\uu_k \right)^n   \right] 
&\leq \frac{\Gamma_{kk}n^2}{2}    C_{k,n-2} +\abs{\Gamma_{kj}}n C_{k,n-1}    + \frac{\Gamma_{jj}}{2}     C_{k,n}+  \sum_{i=1,i\neq k}^{\Kdim}   \abs{R_{ki}}n C_{k,n-1}  +    C_{k,n} .
\end{align}
Therefore, we have $d-1$ inequalities as follows:
\begin{equation*}
\sum_{i=1,i\neq k }^{\Kdim}  R_{ji}  r^{i-1}  \mathbb{E}_{\nu_{i}}\left[ \left( r^k Z^\uu_k \right)^n   \right] \leq c_j, \quad j\in \{1,2,\ldots, \Kdim\}\setminus \{k\},
\end{equation*}
where $c_j$ is defined in \eqref{eq: cj}, which a constant independent of $r$. 
If we denote the column vector $x$ by
$$ x = \left(    \mathbb{E}_{\nu_{1}}\left[ \left( r^k Z^\uu_k \right)^n   \right],  r   \mathbb{E}_{\nu_{2}}\left[ \left( r^k Z^\uu_k \right)^n   \right], \ldots,  r^{\Kdim-1}  \mathbb{E}_{\nu_{\Kdim}}\left[ \left( r^kZ^\uu_k \right)^n   \right]\right)$$ 
and the column vector $ c= (c_1,\ldots, c_\Kdim)$ with $c_k=0$, then the matrix form of the above inequalities is
\begin{equation} \label{eq: matrix form inequality}
R_{SS} x_S \le  c_S,
\end{equation}
where $S=\{1,2,3,\ldots, \Kdim\}\setminus \left\{ k \right\}$, $x_S$ and $c_S$ are, respectively, the subvectors of $x$ and $c$ with indices in $S$, and $R_{SS}$ is a submatrix of $R$ with indices in $S$. Since $R_{SS}$ is a principal submatrix of $R$, $R_{SS}$ is also an $\caM$-matrix. And hence, $R_{SS} x_S \le  c_S$ implies $ x_S \le R_{SS}^{-1} c_S$.
Thus, the second statement of \eqref{eq: moment bound induction} holds for any $i=1,\ldots,d$ with $C_{i,k,n}=(R_{SS}^{-1} c_S)_i$ for $i\in S$ and $C_{k,k,n}=0$.

\subsection{Two-dimensional \texorpdfstring{$\caP$}{P}-matrix Case}\label{subsubsec: 2d P matrix moment}
In this section, we will consider the two-dimensional $\caP$ reflection matrix in Corollary~\ref{prop: 2-d P matrix}. The properties of $\caP$-matrix demonstrate that $R_{11},R_{22}>0$ and $\det(R)=R_{11}R_{22}-R_{12}R_{21}>0$. 

To prove Condition \ref{assmpt: moment bound P} holds for two-dimensional $\caP$-matrix, it is also sufficient to prove that for any $k=1,2$ and any $n\in \Z_+$, there exist positive constants $C_{k,n}, C_{1, k,n}, C_{2, k,n}<\infty$ such that for all $r\in (0,r_0)$,
\begin{align}
	&\E_{\pi} \left[\left(r^k Z_k^\uu\right)^{n}\right] < C_{k,n}, \quad r^{i-1}\E_{\nu_i}\left[ \left( r^{k}Z_k^\uu \right)^n \right] < C_{i,k,n} , \quad  \forall i = 1,2,  \label{eq: moment bound 2d induction} 
\end{align}
where $r_0=\min\{ 1, R_{11}/(2\abs{R_{12}})\}\in (0,1]$ is a constant independent of $k$ and $n$, and the value follows from Lemma~\ref{lmm: eta} for $d=2$.
% The Assumption \ref{assmpt: moment bound P} can be written as \eqref{eq: moment bound P} by introducing the definition of the boundary measure.
The first statement of \eqref{eq: moment bound P} is a direct corollary of the first one in \eqref{eq: moment bound 2d induction} for $n=d=2$. 
The second statement of \eqref{eq: moment bound P} can be proved by Hölder's inequality as follows: for any $r\in (0,r_0)$ and any $i,k\in\{1,2\}$,
\begin{equation*}
r\E_{\nu_i} \left[ \left( r^k Z_k^\uu \right)^d \right] \leq \left( r^i\E_{\nu_i} \left[ \left( r^k Z_k^\uu \right)^{di} \right] \right)^{1/i} \leq \left( r C_{i,k,di} \right)^{1/i} = o(1), \text{ as $r\to 0$.}
\end{equation*}

We first prove $k=1$, then $k=2$. For $k=1$, we use the mathematical induction to prove \eqref{eq: moment bound 2d induction} for $n\in \mathbb{Z}_+$.
As \eqref{eq: moment bound 2d induction} is naturally satisfied for $n=0$. We further assume that \eqref{eq: moment bound 2d induction} holds for any $m=0,1,\ldots,n-1$, and we will now prove it for $n$.

We use the test function $f\left(z \right) = r^{-2} (rz_1)^{n+1}/(n+1)$, which is the same as \eqref{eq: test function moment} for $k=1$. Then, BAR \eqref{eq:BAR} implies
\begin{align*} 
    & \frac{1}{2}  \Gamma_{11} \E_{\pi}\left[ n \left( rZ_1^\uu \right)^{n-1}  \right] -  R_{11} \E_{\pi}\left[ \left( rZ_1^\uu \right)^{n} \right] - r R_{12} \E_{\pi}\left[  \left( rZ_1^\uu \right)^{n}\right] \\
	&\quad + R_{11} \E_{\nu_1}\left[ \left( rZ_1^\uu \right)^{n} \right] + r R_{12} \E_{\nu_2}\left[  \left(rZ_1^\uu \right)^{n} \right] =0
\end{align*}
which implies
\begin{align} 
    \left( R_{11}  + R_{12}r \right) \E_{\pi}\left[  \left( rZ_1^\uu \right)^{n}\right]&= \frac{1}{2}  \Gamma_{11} \E_{\pi}\left[ n \left( rZ_1^\uu \right)^{n-1}  \right] + R_{12}r \E_{\nu_2}\left[ \left( rZ_1^\uu \right)^{n} \right] \notag\\
	& \leq \frac{1}{2}  \Gamma_{11} nC_{1,n-1} + R_{12}r \E_{\nu_2}\left[ \left( rZ_1^\uu \right)^{n} \right]  \label{eq:base3}
\end{align}
We can use another test function $f\left(z \right) = r^{-2} (rz_1)^n \exp(-arz_2)$, where $a$ is a positive constant and will be determined later, which is slightly different from \eqref{eq: test func nu} for $\caM$-matrix. Then, BAR \eqref{eq:BAR} implies
    \begin{align*}
        &\mathbb{E}_{\pi}\left[\frac{1}{2} \Linner{\Gamma, \begin{bmatrix}
            n(n-1) \left( rZ_1^\uu \right)^{n-2} & -an\left( rZ_1^\uu \right)^{n-1} \\ -an\left( rZ_1^\uu \right)^{n-1} & a^2 \left( rZ_1^\uu \right)^{n}
        \end{bmatrix}  } \exp\left( -arZ_2^\uu \right)  \right] \\
        &\quad  -  \E_{\pi}\left[ \left( R_{11} n\left( rZ_1^\uu \right)^{n-1}- R_{21} a\left( rZ_1^\uu \right)^{n} \right) \exp\left( -arZ_2^\uu \right) \right] \\
		&\quad- r \E_{\pi}\left[ \left( R_{12}n\left( rZ_1^\uu \right)^{n-1}-R_{22}a\left( rZ_1^\uu \right)^{n} \right) \exp\left( -arZ_2^\uu \right) \right]\\
        &\quad   + r \E_{\nu_2}\left[   R_{12}n\left( rZ_1^\uu \right)^{n-1}-R_{22}a\left( rZ_1^\uu \right)^{n}  \right] 
          =0,
    \end{align*}
    which implies $r\E_{\nu_2}[   ( rZ_1^\uu )^{n}  ]$ can be bounded by $\mathbb{E}_{\pi}[  ( rZ_1^\uu )^{n} ]$ and other uniformly bounded terms:
    \begin{align}
        r R_{22} \E_{\nu_2}\left[   \left( rZ_1^\uu \right)^{n}  \right] &\leq  \left( \frac{ \Gamma_{22}a}{2} + R_{21} + r R_{22} \right) \mathbb{E}_{\pi}\left[  \left( rZ_1^\uu \right)^{n}    \right]  +\frac{\Gamma_{11}n(n-1)C_{1,n-2}}{2a} \notag \\
		& +\frac{\abs{\Gamma_{12}}an C_{1,n-1}+|R_{12}|nC_{1,n-1}   +      |R_{12}| n C_{2,1,n-1}  }{a}\notag \\
		&\equiv  \left( \frac{ \Gamma_{22}a}{2} + R_{21} + r R_{22} \right) \mathbb{E}_{\pi}\left[  \left( rZ_1^\uu \right)^{n}    \right] + \Theta(a).  \label{eq:base4}
    \end{align}
    If $R_{12}\leq 0$, Lemma \ref{lmm: eta} implies $ R_{11}  + R_{12}r \geq R_{11}/2$, and \eqref{eq:base3} shows that $\E_{\pi}[ (rZ_1^\uu)^n ]$ is uniformly bounded, and hence, \eqref{eq:base4} implies that $r\E_{\nu_2}[ (rZ_1^\uu)^{n} ]$ is also uniformly bounded. Otherwise, combining \eqref{eq:base3} and \eqref{eq:base4}, we obtain 
    \begin{align*} 
        &R_{22}\left( R_{11}  + R_{12}r \right) \E_{\pi}\left[  \left( rZ_1^\uu \right)^{n}\right]\leq  \frac{1}{2}  R_{22}\Gamma_{11} nC_{1,n-1} + R_{12}\left( \frac{ \Gamma_{22}c}{2} + R_{21} + rR_{22} \right) \mathbb{E}_{\pi}\left[  \left( rZ_1^\uu \right)^{n}    \right]   +R_{12}\Theta(a).
    \end{align*}
    which implies  
    \begin{align*}
        \left( R_{11}R_{22}   - \frac{\Gamma_{22}R_{12}c}{2} -R_{12}R_{21} \right) \E_{\pi}\left[  \left( rZ_1^\uu \right)^{n}\right]&\leq   \frac{1}{2}  R_{22}\Gamma_{11} nC_{1,n-1}+\Theta(a).
    \end{align*}
    We can set $a = (R_{11}R_{22}-R_{12}R_{21})/(R_{12}\Gamma_{22})$, then the left-hand side will becomes
	\begin{equation*}
		\frac{1}{2} \left( R_{11}R_{22}  -R_{12}R_{21} \right) \E_{\pi}\left[  \left( rZ_1^\uu \right)^{n}\right],
	\end{equation*}
	which is uniformly bounded by a constant independent of $r$. Since $\det(R)= R_{11}R_{22}  -R_{12}R_{21} >0$, $\E_{\pi}[ ( rZ_1^\uu )^{n}]$ in \eqref{eq: moment bound 2d induction} is uniformly bounded, and hence, \eqref{eq:base4} implies that $r\E_{\nu_2}[ ( rZ_1^\uu )^{n}]$ in \eqref{eq: moment bound 2d induction} is also uniformly bounded.

For $k=2$, the vector $u_2$ in \eqref{eq: u} will be $(-R_{21}/R_{11}, 1)$ with the property $u_2'R^{(1)}=0$ and $u_2'R^{(2)}=\operatorname*{det}(R)/R_{11}>0$. We use the test function $f\left(z \right) = r^{-4} (r^2u^\T z)^{n+1}/(n+1) $, which is the same as \eqref{eq: test function moment} for $k=2$. Then, BAR \eqref{eq:BAR} implies
\begin{align*}
    &\mathbb{E}_{\pi}\left[\frac{n}{2} u_2^\T \Gamma u_2 \left( r^2 u_2^\T Z^\uu \right)^{n-1} \right]   -    \mathbb{E}_{\pi}\left[ \frac{\det(R)}{R_{11}}  \left( r^2 u_2^\T Z^\uu \right)^{n}\right]    +    \mathbb{E}_{{\nu}_{2}}\left[\frac{\det(R)}{R_{11}}\left( -r^2 \frac{R_{21}}{R_{11}} Z^\uu_1 \right)^{n}\right]  =0.
\end{align*}
Hence, \eqref{eq: moment bound 2d induction} for $k=1$ suggests that
\begin{align*}
    \mathbb{E}_{\pi}\left[  \left( r^2 u_2^\T Z^\uu \right)^{n}\right] &  =  \frac{n R_{11} u_2^\T \Gamma u_2}{2\det(R)} \mathbb{E}_{\pi}\left[  \left( r^2 u_2^\T Z^\uu \right)^{n-1} \right]    +   \left( - \frac{R_{21}}{R_{11}} \right)^n r^n \mathbb{E}_{{\nu}_{2}}\left[ \left( r  Z^\uu_1 \right)^{n}\right] \\
	& \leq \frac{n R_{11} u_2^\T \Gamma u_2}{2\det(R)}  \mathbb{E}_{\pi}\left[  \left( r^2 u_2^\T Z^\uu \right)^{n-1} \right]   +   \left| \frac{R_{21}}{R_{11}} \right|^n C_{2,1,n}.
\end{align*}
Iteratively, we have
\begin{align*}
    \mathbb{E}_{\pi}\left[  \left( r^2 u_2^\T Z^\uu \right)^{n}\right] &\leq \frac{n R_{11} u_2^\T \Gamma u_2}{2\det(R)}  \mathbb{E}_{\pi}\left[  \left( r^2 u_2^\T Z^\uu \right)^{n-1} \right]   +   \left| \frac{R_{21}}{R_{11}} \right|^n C_{2,1,n}\\
	& \leq \sum_{m=0}^n C_{2,1,n-m}\left| \frac{R_{21}}{R_{11}} \right|^{n-m}\left( \frac{ R_{11} u_2^\T \Gamma u_2}{2\det(R)}  \right)^{m} \frac{n!}{(n-m)!},
\end{align*}
which implies that $\mathbb{E}_{\pi}[  ( r^2 u_2^\T Z^\uu )^{n}]$ is uniform bounded for any $r\in (0,r_0)$.
Therefore, when $n$ is even, $\E_{\pi}[ (r^2 Z_2^\uu)^n ]$ is also uniformly bounded, because
\begin{align*}
	\E_{\pi}\left[ \left(r^2 Z_2^\uu\right)^n \right] &= \E_{\pi}\left[ \left( r^2 u_2^\T Z^\uu+\frac{R_{21}}{R_{11}}r^2 Z_1^\uu\right)^n \right] \leq 2^n \E_{\pi}\left[ \left| r^2 u_2^\T Z^\uu\right|^n  + \left|\frac{R_{21}}{R_{11}}r^2 Z_1^\uu\right|^n \right]\\
	& \leq 2^n \left( \sum_{m=0}^n C_{2,1,n-m}\left| \frac{R_{21}}{R_{11}} \right|^{n-m}\left( \frac{ R_{11} u_2^\T \Gamma u_2}{2\det(R)}  \right)^{m} \frac{n!}{(n-m)!} + \left|\frac{R_{21}}{R_{11}}\right|^nC_{1,n} \right).
\end{align*}
When $n$ is odd, we know $\E_{\pi}[ (r^2 Z_2^\uu)^{n+1} ]$ is uniformly bounded, and hence Hölder's inequality implies $\E_{\pi}[ (r^2 Z_2^\uu)^n ]$ is also uniformly bounded.

Since $Z_2=0$ almost surely under the boundary measure $\nu_2$, we can obtain that $r\E_{\nu_2}[ ( r^{2}Z_2^\uu )^n ]=0$ and hence, $C_{2,n,2}=0$. The proof for $\E_{\nu_1}[ (r^2 Z_2^\uu)^n ]$ is simpler than $k=1$. We use the test function $f\left(z \right) = r^{-3} (r^2 z_2)^{n} \exp(-r^2z_1)$, which is the same as \eqref{eq: test func nu}.  Hence, BAR \eqref{eq:BAR} becomes
\begin{align*}
    &\mathbb{E}_{\pi}\left[\frac{1}{2} \Linner{\Gamma, \begin{bmatrix}
        r \left( r^2Z_2^\uu \right)^{n}& -nr\left( r^2Z_2^\uu \right)^{n-1} \\ -nr\left( r^2Z_2^\uu \right)^{n-1} &  n(n-1) r \left( r^2Z_2^\uu \right)^{n-2}
    \end{bmatrix}  } \exp\left( -r^2Z_1^\uu \right)  \right] \\
    &\quad  -  \E_{\pi}\left[ \left(-R_{11}\left( r^2Z_2^\uu \right)^{n} + R_{21}n\left( r^2Z_2^\uu \right)^{n-1}  \right) \exp\left( -r^2Z_1^\uu \right) \right] \\
    &\quad- r \E_{\pi}\left[ \left( -R_{12}\left( r^2Z_2^\uu \right)^{n}+nR_{22}\left( r^2Z_2^\uu \right)^{n-1} \right) \exp\left( -r^2Z_1^\uu \right) \right]\\
    &\quad   + \E_{\nu_1}\left[  -R_{11}\left( r^2Z_2^\uu \right)^{n} + R_{21}n\left( r^2Z_2^\uu \right)^{n-1}     \right] 
        =0.
\end{align*}
Therefore, since $\E_{\pi}[ (r^2 Z_2^\uu)^{n} ]$ is uniformly bounded, we have 
\begin{align*}
	\E_{\nu_1}\left[ \left( r^2Z_2^\uu \right)^{n}   \right]  \leq  & \frac{|R_{21}|}{R_{11}}n\E_{\nu_1}\left[ \left( r^2Z_2^\uu \right)^{n-1}   \right] + \frac{
		\Gamma_{11}C_{2,n}+ 2\abs{\Gamma_{12}} nC_{2,n-1}+ \Gamma_{22} n(n-1) C_{2,n-2}   }{2R_{11}}  \\
    &\quad  +   \frac{R_{11}C_{2,n} + \abs{R_{21}}nC_{2,n-1} + \abs{R_{12}}C_{2,n}+nR_{22}C_{2,n-1}  }{R_{11}}    \\
	& \equiv \frac{|R_{21}|}{R_{11}}n\E_{\nu_1}\left[ \left( r^2Z_2^\uu \right)^{n-1}   \right] + \Theta(n).
\end{align*}
Iteratively, we have 
\begin{equation*}
	\E_{\nu_1}\left[ \left( r^2Z_2^\uu \right)^{n}   \right] \leq \sum_{m=0}^n \Theta(n-m) \left( \frac{|R_{21}|}{R_{11}} \right)^{m} \frac{n!}{(n-m)!}.
\end{equation*}
Hence, $\E_{\nu_1}[ (r^2 Z_2^\uu)^n ]$ is uniformly bounded, which completes the proof.

\subsection{Lower Triangular \texorpdfstring{$\caP$}{P}-matrix Case}\label{subsec: FBFS P matrix moment}

In this section, we will consider the lower triangular $\caP$ reflection matrix in Corollary \ref{prop: FBFS reentrant Queue}. To prove Condition \ref{assmpt: moment bound P} holds for lower triangular $\caP$-matrix, it is sufficient to prove that for any $k=1,\ldots, d$ and any $n\in \Z_+$, there exist positive constants $C_{k,n}, C_{1, k,n}, \ldots, C_{d, k,n}<\infty$ such that for all $r\in (0,r_0)$,
\begin{align}
	&\E_{\pi} \left[\left(r^k Z_k^\uu\right)^{n}\right] < C_{k,n}, \quad r^{i-1}\E_{\nu_i}\left[ \left( r^{k}Z_k^\uu \right)^n \right] < C_{i,k,n} , \quad  \forall\ i = \{ 1, \ldots, d\},  \label{eq: moment bound FBFS induction} 
\end{align}
where $r_0\in (0,1]$ is a constant independent of $k$ and $n$, and the value follows from Lemma~\ref{lmm: eta}. 
The first statement of \eqref{eq: moment bound P} is a direct corollary of the first one in \eqref{eq: moment bound FBFS induction} for $n=d$. 
The second statement of \eqref{eq: moment bound P} can be proved by Hölder's inequality as follows: for any $r\in (0,r_0)$ and any $i,k\in\{1,\ldots, d\}$,
\begin{equation*}
r\E_{\nu_i} \left[ \left( r^k Z_k^\uu \right)^d \right] \leq \left( r^i\E_{\nu_i} \left[ \left( r^k Z_k^\uu \right)^{di} \right] \right)^{1/i} \leq \left( r C_{i,k,di} \right)^{1/i} = o(1), \text{ as $r\to 0$.}
\end{equation*}

Our proof also utilizes the vectors $\{u_k\}_{k=1}^d$ defined in \eqref{eq: u}, but the lower triangular matrix can lead to more properties than $\caM$-matrix in \eqref{eq: w}, Lemma \ref{lmm: uk} and \eqref{eq: uk neg}:
\begin{equation} \label{eq: u triangular}
	u_k'R_{:,i}=0, \text{ for }i<k,~ u_k'R_{:,k}>0, \text{ and triangular $R$ implies } u_k'R_{:,i}= 0, \text{ for } i>k.
\end{equation}
Although $u_k$ may have negative entries, $u_k'R_{:,i}= 0$ for $i>k$ can make the following proof similar and even easier than the proof of $\caM$-matrix in Section \ref{subsec: uniform moment bound}.

Suppose that \eqref{eq: moment bound FBFS induction} holds for any $\ell=1,\ldots,k-1$, and we will proceed to prove it for $k$. To prove the first statement in \eqref{eq: moment bound FBFS induction}, we use the test function as follows, 
\begin{equation*} %\label{eq: test function moment 1}
f_{k,n}\left( z \right)=\frac{r^{-2k}}{n+1}\left( r^k u_k^\T z\right)^{n+1},
\end{equation*}
which is the same as $\caM$-matrix cases in \eqref{eq: test function moment}. 
To shorten the notation, we will similarly write $y_k$ for $r^k u_k^\T z$, and then $\nabla_z f_{k,n}(z) = r^{-k}y_k^nu_k$ and $\nabla^2_z f_{k,n}(z) = ny_k^{n-1}u_ku_k^\T$. We also have the relation 
\begin{equation} \label{eq: yk 1}
	y_k^n  \leq k^n   (r^k z_k)^n +k^n  \sum_{j=1}^{k-1} (r^k {|w_{jk}|} z_j)^n .
\end{equation}
Then, BAR \eqref{eq:BAR} and \eqref{eq: u triangular} imply that
\begin{equation*} %\label{eq: BAR moment bound 1}
\frac{n}{2} u_k^\T\Gamma  u_k \mathbb{E}_{\pi}\left[ \left( Y^\uu_k \right)^{n-1}  \right] - {u_k^\T R_{:,k}} \mathbb{E}_{\pi}\left[  \left( Y^\uu_k \right)^{n}\right]    +   {u_k^\T R_{:,k}}\mathbb{E}_{\nu_k}\left[   \left(  Y^\uu_k \right)^{n}\right]     =0,
\end{equation*}
where we only has $u_k^\T R_{:,k}$ term, because ${u_k^\T R_{:,i}}=0$ for $i\neq k$.
Then \eqref{eq: yk 1} implies
\begin{align*}
	\mathbb{E}_{\pi}\left[  \left( Y^\uu_k \right)^{n}\right] &= \frac{nu_k^\T\Gamma  u_k }{2u_k^\T R_{:,k}}\mathbb{E}_{\pi}\left[ \left( Y^\uu_k \right)^{n-1}  \right] + \mathbb{E}_{\nu_k}\left[   \left(  Y^\uu_k \right)^{n}\right] \\
	&\leq \frac{nu_k^\T\Gamma  u_k }{2u_k^\T R_{:,k}}\mathbb{E}_{\pi}\left[ \left( Y^\uu_k \right)^{n-1}  \right] + \mathbb{E}_{\nu_k}\left[ k^n  \sum_{j=1}^{k-1} \left(r^k {|w_{jk}|} Z_j^\uu\right)^n\right]\\
	&\leq \frac{nu_k^\T\Gamma  u_k }{2u_k^\T R_{:,k}}\mathbb{E}_{\pi}\left[ \left( Y^\uu_k \right)^{n-1}  \right] + k^n  \sum_{j=1}^{k-1}|w_{jk}|^n  C_{k, j, n},
\end{align*}
where the first inequality follows from the relation \eqref{eq: yk 1} and that $Z_k=0$ almost surely under the boundary measure $\nu_k$, and the last inequality holds due to the induction hypotheses.
Iteratively, we have
\begin{align*}
	\mathbb{E}_{\pi}\left[  \left( Y^\uu_k \right)^{n}\right] &\leq  \sum_{m=0}^n \frac{n!}{(n-m)!} \left( \frac{ u_k^\T\Gamma  u_k }{2u_k^\T R_{:,k}}  \right)^{m}  k^{n-m}  \sum_{j=1}^{k-1}|w_{jk}|^{n-m}  C_{k, j, n-m}.
\end{align*}
Hence, $\mathbb{E}_{\pi}[  ( Y^\uu_k )^{n}]$ is uniformly bounded. Therefore, when $n$ is even, $\E_{\pi}[ (r^k Z_k^\uu)^n ]$ is also uniformly bounded, because
\begin{align*}
	\E_{\pi}\left[ \left(r^k Z_k^\uu\right)^n \right] &= \E_{\pi}\left[ \left( r^k u_k^\T Z^\uu-\sum_{j=1}^{k-1} r^k w_{jk} Z^\uu_j\right)^n \right]  \leq k^n \E_{\pi}\left[ \left( r^k u_k^\T Z^\uu\right)^n+\sum_{j=1}^{k-1} \left( r^k w_{jk} Z^\uu_j\right)^n \right]\\
	& \leq k^n \left( \sum_{m=0}^n \frac{n!}{(n-m)!} \left( \frac{ u_k^\T\Gamma  u_k }{2u_k^\T R_{:,k}}  \right)^{m}  k^{n-m}  \sum_{j=1}^{k-1}|w_{jk}|^{n-m}  C_{k, j, n-m} +\sum_{j=1}^{k-1} w_{jk}^n C_{j,n}  \right).
\end{align*}
When $n$ is odd, we know $\E_{\pi}[ (r^k Z_k^\uu)^{n+1} ]$ is uniformly bounded, and hence, Hölder's inequality implies $\E_{\pi}[ (r^k Z_k^\uu)^n ]$ is also uniformly bounded.

To prove the second statement of \eqref{eq: moment bound FBFS induction}, we need to further use mathematical induction for $n\in \mathbb{Z}_+$. We assume that the second statement of \eqref{eq: moment bound FBFS induction} holds for any $m=0,1,\ldots,n-1$, and we will proceed to prove it for $n$. We also apply the test function as follows
\begin{equation*}
f_{j,k,n}\left( z \right)=r^{-k-1}\left( r^k z_k \right)^n\exp\left( -r^k z_j \right)\qquad  j\in \{1,2,\ldots, \Kdim\}\setminus \{k\},
\end{equation*}
which is the same as $\caM$-matrix cases in  \eqref{eq: test func nu}.
Substituting $f_{j,k,n}$ into BAR \eqref{eq:BAR}, we have
\begin{align*}
&\frac{r^{k-1}}{2}\mathbb{E}_{\pi}\left[  \left( \Gamma_{kk}n(n-1) \left( r^k Z^\uu_k \right)^{n-2} - 2 \Gamma_{kj} n  \left( r^k Z^\uu_k \right)^{n-1} + \Gamma_{jj}  \left( r^k Z^\uu_k \right)^{n} \right) \exp\left( -r^k Z^\uu_j \right) \right]  \notag \\
&\quad  -  \mathbb{E}_{\pi}\left[ \left( {\sum_{i=1}^{k} r^{i-1}} R_{ki} n  \left( r^k Z^\uu_k \right)^{n-1}  - {\sum_{i=1}^{j} r^{i-1}} R_{ji}\left( r^k Z^\uu_k \right)^n \right)\exp\left( -r^k Z^\uu_j \right)\right] \notag \\
&\quad  +  {\sum_{i=1}^{k-1} r^{i-1}}  \mathbb{E}_{\nu_{i}}\left[ R_{ki} n  \left( r^k Z^\uu_k \right)^{n-1}   \exp\left( -r^k Z^\uu_j \right)\right] \notag\\
&\quad - {\sum_{i=1,i\neq k}^{j-1} r^{i-1}}  \mathbb{E}_{\nu_{i}}\left[R_{ji}\left( r^k Z^\uu_k \right)^n \exp\left( -r^k Z^\uu_j \right)\right] -  r^{j-1}  \mathbb{E}_{\nu_{j}}\left[R_{jj}\left( r^k Z^\uu_k \right)^n \right]  =0, % \label{eq: BAR moment nu P FBFS}
\end{align*}
where we use the fact that $R_{\ell i}=0$ for $\ell < i$. Therefore, induction hypotheses imply that
\begin{align*}
	&R_{jj}r^{j-1}  \mathbb{E}_{\nu_{j}}\left[\left( r^k Z^\uu_k \right)^n \right] =\frac{1}{2}   \left( \Gamma_{kk}n(n-1) C_{k,n-2} + 2 \abs{\Gamma_{kj}} n  C_{k,n-1} + \Gamma_{jj}  C_{k,n} \right)     \\
	&\quad  +    \sum_{i=1}^{k}  \abs{R_{ki}} n C_{k,n-1}  + \sum_{i=1}^{j}  \abs{R_{ji}}C_{k,n}    +  \sum_{i=1}^{k-1}    \abs{R_{ki}} n  C_{i,k,n-1}  + \sum_{i=1,i\neq k}^{j-1} \abs{R_{ji}} r^{i-1} \mathbb{E}_{\nu_{i}}\left[\left( r^k Z^\uu_k \right)^n \right] , 
\end{align*}
Since $R_{jj} r^{j-1}  \mathbb{E}_{\nu_{j}}[( r^k Z^\uu_k )^n ]$ only depends on $r^{i-1} \mathbb{E}_{\nu_{i}}[( r^k Z^\uu_k )^n]$ for $i\in \{1,\ldots, j-1\}\setminus \{k\}$, we can iteratively obtain the uniform bounds for $r^{j-1} \mathbb{E}_{\nu_{j}}[( r^k Z^\uu_k )^n]$ from $j=1$ to $d$ skipping $k$.

\subsection{Truncated Test Functions} \label{subsec: truncate}
To make the proofs of Sections~\ref{subsec: uniform moment bound}-\ref{subsec: FBFS P matrix moment} more rigorous, we need to replace the unbounded test functions with their truncated counterparts in $C_b^2(\R_+^d)$. Initially, we will introduce the truncated test functions. Subsequently, we will demonstrate in \eqref{eq: truncation property} that these truncated test functions are "similar" to the original unbounded test functions in the sense that they also satisfy the desirable properties. This allows us to assert that the derivations in Sections \ref{subsec: uniform moment bound}-\ref{subsec: FBFS P matrix moment} remain valid even when applied to the truncated test functions. For illustrative purposes, we focus on the proof of Section \ref{subsec: uniform moment bound}, since the test functions in Sections \ref{subsubsec: 2d P matrix moment} and \ref{subsec: FBFS P matrix moment} are similar, and hence, follow a similar rationale.

\paragraph{Truncated test functions.} We employ a ``soft truncation" incorporating an exponentially decaying function. Specifically, the soft-truncated version of the polynomial functions $z^n$, indexed by $\alpha >0$, is given by
\begin{equation*}
	g^\fk_{n}\left( z \right)=z^ne^{- z/\alpha },
\end{equation*}
and hence, the corresponding antiderivative, derivative, and second derivative are given by
\begin{equation*}
	G^\fk_{n}\left( z \right)=\int_0^z x^ne^{- x/\alpha }dx, \quad \quad  \dot{g}_n^\fk(z)=nz^{n-1}e^{- z/\alpha }-z^ne^{- z/\alpha }/\alpha,
\end{equation*}
and
\begin{equation*}
	\ddot{g}_n^\fk(z) = n(n-1)z^{n-2}e^{- z/\alpha }-2nz^{n-1}e^{- z/\alpha }/\alpha+z^ne^{- z/\alpha }/\alpha^2.
\end{equation*}
Therefore, all the above functions are in $C_b^2(\R_+^d)$.
As $x\exp(-x)\leq 1$ and $x^2\exp(-x)\leq  1$, for all $x\geq 0$, we have the following property: for any $\alpha>0$ and $z\geq 0$,
\begin{equation} \label{eq: truncation property}
  0\leq {g}_n^\fk(z) \leq z^{n}, \quad \abs{\dot{g}_n^\fk(x)} \leq (n+1)x^{n-1}, \quad \abs{\ddot{g}_n^\fk(x)} \leq (n+1)^2x^{n-2}.
\end{equation}
In summary, the truncated test functions, corresponding to \eqref{eq: test function moment} and \eqref{eq: test func nu}, are as follows:
\begin{align}
	f^\fk_{k,n}\left( z \right)&=r^{-2k} G^\fk_n\left(y_k \right) = r^{-2k} G^\fk_n\left( r^{k} u_k^\T z \right) \label{eq: test func truncation},\\
	f^\fk_{j,k,n}\left( z \right)&=r^{-k-1}g^\fk_n\left( r^{k} z_k \right)\exp\left( -r^k z_j \right)\qquad  j\in \{1,2,\ldots, \Kdim\}\setminus \{k\}. \label{eq: test func nu truncation}
\end{align}

\paragraph{Applying truncated test functions to BAR.} We next describe the necessary adjustments in the proof of Section \ref{subsec: uniform moment bound} when replacing unbounded test functions with truncated versions. A comparative analysis is provided below.

To prove the first statement in \eqref{eq: moment bound induction}, we apply the test function in \eqref{eq: test func truncation} to BAR \eqref{eq:BAR} and obtain
\begin{equation*}
	\frac{n}{2} u_k^\T\Gamma  u_k \mathbb{E}_{\pi}\left[\dot{g}_n^\fk \left( Y^\uu_k \right)  \right] - \sum_{i=1}^{\Kdim} r^{i-k} u_k^\T R_{:,i} \mathbb{E}_{\pi}\left[  {g}_n^\fk\left( Y^\uu_k \right)\right]    +   \sum_{i=1}^{\Kdim} r^{i-k}  u_k^\T R_{:,i}\mathbb{E}_{\nu_i}\left[  {g}_n^\fk \left(  Y^\uu_k \right)\right]     =0,
\end{equation*}
which is analogous to \eqref{eq: BAR moment bound}. Using the similar computation as in \eqref{eq: moment computation} and \eqref{eq: moment bound res}, we can prove that $\mathbb{E}_{\pi}[  g_n^\fk (Y^\uu_k) ] $ is uniformly bounded for any $r\in(0,r_0)$ and any $\alpha>0$. Finally, as $g^\fk_n\left( z \right) \uparrow z^n$ as $\alpha\to \infty$, the monotone convergence theorem can be applied to obtain the same result in \eqref{eq: moment bound res}, which completes the proof of the first statement in \eqref{eq: moment bound induction}.

To prove the second statement in \eqref{eq: moment bound induction}, we apply the test function in \eqref{eq: test func nu truncation} to BAR \eqref{eq:BAR} and obtain
\begin{align*}
  &\frac{r^{k-1}}{2}\mathbb{E}_{\pi}\left[  \left( \Gamma_{kk} \ddot{g}_n^\fk\left( r^k Z^\uu_k \right)  - 2 \Gamma_{kj} \dot{g}_n^\fk\left( r^k Z^\uu_k \right) + \Gamma_{jj}  {g}_n^\fk\left( r^k Z^\uu_k \right) \right) \exp\left( -r^k Z^\uu_j \right) \right]  \\
  &\quad  - \sum_{i=1}^{\Kdim} r^{i-1} \mathbb{E}_{\pi}\left[ \left( R_{ki} \dot{g}_n^\fk\left( r^k Z^\uu_k \right)  -R_{ji}{g}_n^\fk\left( r^k Z^\uu_k \right) \right)\exp\left( -r^k Z^\uu_j \right)\right]  \\
  &\quad  +  \sum_{i=1,i\neq k}^{\Kdim} r^{i-1}  \mathbb{E}_{\nu_{i}}\left[\left( R_{ki} \dot{g}_n^\fk \left( r^k Z^\uu_k \right)  -R_{ji}{g}_n^\fk\left( r^k Z^\uu_k \right) \right)\exp\left( -r^k Z^\uu_j \right)\right]  =0,
\end{align*}
which is analogous to \eqref{eq: BAR moment nu}. Using the similar computation as in \eqref{eq: moment computation nu} and \eqref{eq: cj}, we can similarly obtain $\sum_{i=1,i\neq k }^{\Kdim}  R_{ji}  r^{i-1}  \mathbb{E}_{\nu_{i}}[ g_n^\fk  ( r^k Z^\uu_k )   ] $ is uniformly bounded for any $r\in(0,r_0)$, any $\alpha>0$ and any $j\in \{1,2,\ldots, \Kdim\}\setminus \{k\}$, and hence, following the same argument as \eqref{eq: matrix form inequality}, we can prove that $\mathbb{E}_{\nu_i}[ g_n^\fk ( r^k Z^\uu_k )  ]$ is uniformly bounded for any $r\in(0,r_0)$ and any $\kappa>0$. Finally, as $g^\fk_n\left( z \right) \uparrow z^n$ as $\alpha\to \infty$, the monotone convergence theorem can be applied to obtain the same result, which completes the proof of the second statement in \eqref{eq: moment bound induction}.

\bibliographystyle{plainnat} % or try plainnat, abbrvnat or unsrtnat
\bibliography{dai20230220_v2} % refers to example.bib
% \newpage

\newpage
\appendix

\section{Proof of Remark and Technical Lemmas}\label{subsec: proof lemmas}

\begin{proof}[Proof of Remark \ref{rmk: skew}]
	Suppose $\Gamma_{kk}=1$ and $R_{kk}=1$ for all $k=1,2,\ldots,d$.  If the family of SRBMs in our setting satisfies the skew symmetric condition in \eqref{eq: skew symmetric}, then Theorem \ref{thm: matrix P} indicates that the limit of $(\delta^\uu_1 Z_1^\uu, \delta^\uu_2 Z_2^\uu, \ldots, \delta^\uu_d Z_d^\uu)$ has a product-form exponential stationary distribution with mean $m_k = u_k^\T\Gamma u_k / (2u_k^\T R u_k) = 1/2$ for $k$th component, where the equalities follow from \eqref{eq: w} and \eqref{eq: skew symmetric}. Hence, the approximation of the stationary distribution of the SRBM $Z^\uu$ is a product-form exponential distribution with mean $1/(2\delta_k^\uu)$ for $k$th component, which recovers the result in skew symmetric case.

	When $\Gamma_{kk}\neq 1$ or $R_{kk}\neq 1$ for some $k=1,2,\ldots,d$, we can use the standardization idea in \citet{DaiHarr1992} to conclude the same result. Proposition 8 of \citet{DaiHarr1992} shows that the standardized SRBM $Z^\#$ with parameters $\Gamma^\#=\Lambda \Gamma \Lambda$, $R^\#=\Lambda R V \Lambda^{-1}$ and $\mu^\# = \Lambda \mu$ has the same distribution as $\Lambda Z$ and the skew symmetric condition \eqref{eq: skew symmetric} is $2\Gamma^\# = R^\# + (R^{\#})^{\T}$. Since Theorem \ref{thm: matrix P} recovers the result in skew symmetric case for the standardized SRBM $Z^\#$, it also holds for the original SRBM $Z$.
\end{proof}

\begin{proof}[Proof of Lemma \ref{lmm: uk}]
	We first introduce the following notation. 
	For $A, B \subseteq \{1,\ldots,d\}$, $R_{A, B}$ denotes the submatrix of $R$ whose entries are $\left(R_{jk}\right)$ for $j \in A$ and $k \in B$, adhering to the conventions below: for $j \leq k$,
	$$
	[j: k]=\{j, j+1, \ldots, k\}, \quad R_{A, k}=R_{A,\{k\}}, \quad R_{k, B}=R_{\{k\}, B}, \quad R_k=R_{[1: k],[1: k]}.
	$$
	Then \eqref{eq: w} be represented in matrix form as
	\begin{equation*}
		[w_{1k}, \ldots, w_{k-1,k}]R_{k-1}+ R_{k,[1:k-1]}=0,
	\end{equation*} 
	and hence,
	\begin{align*} %\label{eq: schur}
	u_k^\T R_{:,k} = [w_{1k}, \ldots, w_{k-1,k}] R_{[1:k-1], k} + R_{kk} = R_{kk} - R_{k,[1:k-1]}^\T R^{-1}_{k-1} R_{[1:k-1], k} 
	\end{align*}
	which is the Schur complement of $R_k$ with respect to $R_{k-1}$, and hence, we have $u_k^\T R_{:,k}=\det(R_k)/\det(R_{k-1})$ \citep{zhan2006}. Since $R$ is an $\caP$-matrix, we have $\operatorname{det}\left( R_{j} \right)>0$ for $j=1,\ldots,d$, and hence, $u_k^\T R_{:,k}>0$.
\end{proof}

The proofs of Lemmas \ref{lmm:SSCs}, \ref{lmm:SSCs P} and \ref{lmm:extra} depend on the uniform moment bounds in \eqref{eq: moment bound} or the boundary measure form \eqref{eq: moment bound P}.
Since Lemma \ref{lmm:SSCs P} can degenerate to Lemma \ref{lmm:SSCs} when $R$ is an $\caM$-matrix, we will only provide the proof of Lemma \ref{lmm:SSCs P} in this section.

\begin{proof}[Proof of Lemma \ref{lmm:SSCs P}]
	To shorten the notation, we define the vector $\beta^\uu$ as follows:
	$$\beta^\uu=\theta^\uu_k-(0,...,0,r^k\eta_k,  r^{k+1}\eta_{k+1},...,r^\Kdim \eta_\Kdim)^\T.$$
	Then following the definition of $\theta^\uu_k$ in \eqref{eq: theta}, we can analyze the elements of $\beta^\uu$ as follows:
	\begin{equation*}
		\beta_j^\uu = \begin{cases}
			O(r^k) ,&\text{ for }j\leq k-1, \\
			O(r^{j+1}),&\text{ for }j\geq k.
		\end{cases}
	\end{equation*}

	Condition \ref{assmpt: moment bound P} implies that ${r^{j+1/2}} Z_j^\uu \overset{p}{\to} 0$ for any $j=1,\ldots,d$ and for all stationary distributions $\{\pi, \nu_1, \ldots, \nu_d\}$, because Markov's inequality implies that for any $\epsilon>0$,
	\begin{equation*}
		\lim_{r\downarrow 0} \mathbb{P}\left( \abs{r^{j+1/2} Z_j^\uu} > \epsilon \right) \leq \lim_{r\downarrow 0} \frac{\E\left[ \abs{r^{j+1/2} Z_j^\uu}^2  \right]}{\epsilon^2} = \lim_{r\downarrow 0} \frac{r\E\left[ \left( r^{j} Z_j^\uu \right)^2  \right]}{\epsilon^2} = 0.
	\end{equation*}
	Therefore, we have $\theta^{(r)}_{k,j} Z_j^\uu=\beta_j^\uu Z_j^\uu  \overset{p}{\to} 0$ for $j<k$, where $\theta^{(r)}_{k,j}$ denotes the $j$th element of the vector $\theta^{(r)}_{k}$. For $j\geq k$, 
	\begin{align*}
		\kappa\left( \theta^{(r)}_{k,j} Z_j^\uu \right)  - \eta_jr^j Z_j^\uu&=\kappa\left(\eta_jr^j Z_j^\uu+ \beta^{(r)}_j Z_j^\uu \right)  - \kappa \left( \eta_jr^j Z_j^\uu \right) \\
		&= \beta^{(r)}_j Z_j^\uu \dot\kappa\left(\eta_jr^j Z_j^\uu+ \xi \beta^{(r)}_j Z_j^\uu \right) \overset{p}{\to} 0,
	\end{align*}
	where $\xi$ is a random variable taking value in $(0,1)$ by the mean value theorem, and the convergence in probability follows from $\beta^{(r)}_j=O(r^{j+1})$ and $\dot \kappa$ is bounded. Hence, the bounded convergence theorem and continuous mapping theorem imply that \eqref{eq: SSC A P} becomes
    \begin{align*}
        &\lim_{r\downarrow 0}\abs{\psi^\uu(\theta^{(r)}_k)-\phi^\uu(0,...,0,r^k\eta_k,r^{k+1}\eta_{k+1},  ...,r^\Kdim \eta_\Kdim)} \\
        &\quad \leq \lim_{r\downarrow 0}\abs{\E_{\pi}\left[ e^{\sum_{j<k}\kappa\left( \theta^{(r)}_{k,j} Z_j^\uu \right) + \sum_{j\geq k} \left( \kappa\left( \theta^{(r)}_{k,j} Z_j^\uu \right)  - \eta_jr^j Z_j^\uu \right) }\right] -1} \\
		&\quad =\abs{\E_{\pi}\left[ e^{\sum_{j<k}\lim\limits_{r\downarrow 0}\kappa\left( \theta^{(r)}_{k,j} Z_j^\uu \right) + \sum_{j\geq k} \lim\limits_{r\downarrow 0}\left( \kappa\left( \theta^{(r)}_{k,j} Z_j^\uu \right)  - \eta_jr^j Z_j^\uu \right) }\right] -1}  =0.
    \end{align*}
	Similarly, since $Z_i=0$ almost surely if $Z$ follows the probability measure $\nu_i$, the bounded convergence theorem and continuous mapping theorem imply that \eqref{eq: SSC B P} becomes
	\begin{align*}
	    &\lim_{r\downarrow 0}\abs{\psi^\uu_k(\theta^{(r)}_k)-\phi^\uu_k(0,...,0,0,r^{k+1}\eta_{k+1},  ...,r^\Kdim \eta_\Kdim)} \\
	    &\quad \leq \lim_{r\downarrow 0}\abs{\E_{\nu_i}\left[ e^{\sum_{j<k}\kappa\left( \theta^{(r)}_{k,j} Z_j^\uu \right) + \sum_{j> k} \left( \kappa\left( \theta^{(r)}_{k,j} Z_j^\uu \right)  - \eta_jr^j Z_j^\uu \right) }\right] -1} \\
		&\quad = \abs{\E_{\nu_i}\left[ e^{\sum_{j<k}\lim\limits_{r\downarrow 0}\kappa\left( \theta^{(r)}_{k,j} Z_j^\uu \right) + \sum_{j> k} \lim\limits_{r\downarrow 0}\left( \kappa\left( \theta^{(r)}_{k,j} Z_j^\uu \right)  - \eta_jr^j Z_j^\uu \right) }\right] -1}  =0.
	\end{align*}
	The proofs of \eqref{eq: SSC C P} and  \eqref{eq: SSC D P} follow from a similar argument, so we omit the detail.
\end{proof}

\begin{proof}[Proof of Lemma \ref{lmm:extra}]
	To prove $\gamma^\uu(\theta^{(r)}_k) = o(r^{2k})$, it suffices to show that for any $i,j\in \{1,...,d\}$,
	\begin{align}
		&\E_{\pi}\left[  \epsilon_{g,\theta^\uu_k,j}\left( Z^\uu \right)  \tilde{f}_{\theta^{(r)}_k}\left( Z^\uu \right)\right] = o(r^{2k-1}), \label{eq:extra1-1}\\
		&\E_{\nu_i}\left[  \epsilon_{g,\theta^\uu_k,j}\left( Z^\uu \right)  \tilde{f}_{\theta^{(r)}_k}\left( Z^\uu \right)\right] = o(r^{2k-1}),\label{eq:extra1-2}\\
		&\E_{\pi}\left[\epsilon_{H,\theta^\uu_k,ij}\left( Z^\uu \right)\tilde{f}_{\theta^{(r)}_k}\left( Z^\uu \right)\right]=o(r^{2k}),\label{eq:extra1-3}
	\end{align}
	where the extra terms $\epsilon_{g,\theta}$ and $\epsilon_{H,\theta}$ is defined in truncated MGF-BAR \eqref{eq:TMGF-BAR}.

	By the definition of $\theta^\uu_k$ in \eqref{eq: theta}, we have $\theta^{(r)}_{k,j} = r^j\eta_j + \sum_{\ell=j+1}^d r^\ell\eta_\ell w_{j\ell}$ for $j\geq k$. Since $\eta_{k,j} < 0$, we have $\theta^{(r)}_j < 0$ for $r$ small enough, and then \eqref{eq: eps0} implies that \eqref{eq:extra1-1}-\eqref{eq:extra1-3} hold for any $i,j\geq k$.

	For $j<k$, the definition of $\theta_k^\uu$ in \eqref{eq: theta} shows that $\theta_{k,j}^\uu = O(r^k)$, and hence, we have
	\begin{align}
		\E_{\pi}\left[  \I{\theta^\uu_{k,j} Z^\uu_j> 1}  \right] &\leq   \E_{\pi}\left[  \I{\abs{\theta^\uu_{k,j} Z^\uu_j}> 1}  \right]  \leq  \E_{\pi}\left[  \abs{\theta^\uu_{k,j} Z^\uu_j}^d  \right]\notag \\
		& \leq   \left( {\abs{\theta^\uu_{k,j}}}/{r^{j}} \right)^d \E_{\pi}\left[ \left(r^j Z^\uu_j\right)^d  \right]  =  O(r^{(k-j)d}) O(1) =O(r^{d}), \label{eq: eps}
	\end{align}
	where the second inequality holds due to Markov's inequality. Therefore, \eqref{eq:extra1-1} becomes
	\begin{align*}
		&\abs{\E_{\pi}\left[  \epsilon_{g,\theta^\uu_k,j}\left( Z^\uu \right)  \tilde{f}_{\theta^{(r)}_k}\left( Z^\uu \right)\right]} =  \abs{\E_{\pi}\left[ \theta^\uu_{k,j} \left( \dot{\kappa}\left(\theta^\uu_{k,j} Z^\uu_j\right) - 1\right)  \tilde{f}_{\theta^{(r)}_k}\left( Z^\uu \right)\right]}\\
		&\quad \leq \E_{\pi}\left[ \abs{\theta^\uu_{k,j}} \abs{ \dot{\kappa}\left(\theta^\uu_{k,j} Z^\uu_j\right) - 1} \abs{ \tilde{f}_{\theta^{(r)}_k}\left( Z^\uu \right)}\right] \leq \abs{\theta^\uu_{k,j}} \E_{\pi}\left[  \I{\theta^\uu_{k,j} Z^\uu_j> 1}  \right]e^{2d}\\
		&\quad = O(r^{k}) O(r^d) O(1) =O(r^{k+d})= o(r^{2k-1}),
	\end{align*}
	where the second inequality follows from \eqref{eq:kappa properties}. Similarly, for $i=1,\ldots,d$, \eqref{eq:extra1-2} becomes
	\begin{align*}
		\abs{\E_{\nu_i}\left[  \epsilon_{g,\theta^\uu_k,j}\left( Z^\uu \right)  \tilde{f}_{\theta^{(r)}_k}\left( Z^\uu \right)\right] }&\leq \abs{\theta^\uu_{k,j}} \left( {\abs{\theta^\uu_{k,j}}}/{r^{j}} \right)^d \E_{\nu_i}\left[ \left(r^j Z^\uu_j\right)^d  \right]e^{2d}\\
		&= O(r^{k}) O(r^{(k-j)d}) o(1/r)O(1) = o(r^{2k-1}).
	\end{align*}
	For $i<k$ and $j<k$, the definition of $\theta_k^\uu$ in \eqref{eq: theta} shows that $\theta_{k,i} = O(r^k)$ and $\theta_{k,j} = O(r^k)$, and hence, \eqref{eq:extra1-3} becomes
	\begin{align}
		&\abs{\E_{\pi}\left[\epsilon_{H,\theta^\uu_k,ij}\left( Z^\uu \right)\tilde{f}_{\theta^{(r)}_k}\left( Z^\uu \right)\right]}\notag\\
		&  = \abs{\E_{\pi}\left[\left( \theta_{k,i}^\uu \theta_{k,j}^\uu \left(\dot{\kappa}\left(\theta_{k,i}^\uu Z^\uu_i\right)\dot{\kappa}\left(\theta^\uu_{k,j} Z^\uu_j\right) - 1\right)+\left(\theta_{k,i}^\uu\right)^2\ddot{\kappa}\left(\theta_{k,i}^\uu Z^\uu_i\right)\I{i=j} \right)\tilde{f}_{\theta^{(r)}_k}\left( Z^\uu \right)\right]}\notag\\
		&  \leq  \abs{\theta_{k,i}^\uu}  \abs{\theta_{k,j}^\uu}  \E_{\pi}\left[\abs{\dot{\kappa}\left(\theta_{k,i}^\uu Z^\uu_i\right)\dot{\kappa}\left(\theta^\uu_{k,j} Z^\uu_j\right) - 1}\right]e^{2d}+ \abs{\theta_i^\uu}^2 \E_{\pi}\left[\abs{\ddot{\kappa}\left(\theta_{k,i}^\uu Z^\uu_i\right)} \right] e^{2d}, \label{eq: epsH}
	\end{align}
	where the inequality follows from \eqref{eq:kappa properties}. The first term in \eqref{eq: epsH} becomes
	\begin{align}
		&\E_{\pi}\left[\abs{\dot{\kappa}\left(\theta_{k,i}^\uu Z^\uu_i\right)\dot{\kappa}\left(\theta^\uu_{k,j} Z^\uu_j\right) - 1}\right]\notag \\
		& \quad = \E_{\pi}\left[\abs{\dot{\kappa}\left(\theta_{k,i}^\uu Z^\uu_i\right)\left( \dot{\kappa}\left(\theta^\uu_{k,j} Z^\uu_j\right) -  1 \right) + \dot{\kappa}\left(\theta_{k,i}^\uu Z^\uu_i\right) - 1}\right]\notag\\
		&\quad \leq \E_{\pi}\left[\abs{\dot{\kappa}\left(\theta_{k,i}^\uu Z^\uu_i\right)}\abs{ \dot{\kappa}\left(\theta^\uu_{k,j} Z^\uu_j\right) -  1 } + \abs{\dot{\kappa}\left(\theta_{k,i}^\uu Z^\uu_i\right) - 1}\right]\notag\\
		&\quad \leq 2 \E_{\pi}\left[ \I{\theta^\uu_{k,j} Z^\uu_j> 1}  \right]+\E_{\pi}\left[  \I{\theta^\uu_{k,i} Z^\uu_i> 1} \right] = O(r^{d}), \label{eq: epsH1}
	\end{align}
	where the second inequality follows from \eqref{eq:kappa properties} and the last equality holds due to \eqref{eq: eps}.
	The second term in \eqref{eq: epsH} becomes
	\begin{align}
		&\E_{\pi}\left[\abs{\ddot{\kappa}\left(\theta_{k,i}^\uu Z^\uu_i\right)} \right] \leq \E_{\pi}\left[4\I{1< \theta_{k,i}^\uu Z^\uu_i< 2} \right] \leq 4\E_{\pi}\left[ \I{\theta^\uu_{k,j} Z^\uu_j> 1}  \right] = O(r^{d}), \label{eq: epsH2}
	\end{align}
	where the first inequality follows from \eqref{eq:kappa properties} and the last equality holds due to \eqref{eq: eps}. Combining \eqref{eq: epsH1} and \eqref{eq: epsH2}, \eqref{eq: epsH} becomes
	\begin{equation*}
		\abs{\E_{\pi}\left[\epsilon_{H,\theta^\uu_k,ij}\left( Z^\uu \right)\tilde{f}_{\theta^{(r)}_k}\left( Z^\uu \right)\right]} = O(r^k) O(r^k) O(r^d) O(1) + O(r^{2k}) O(r^d) O(1) = o(r^{2k}).
	\end{equation*}
	For $i,j$ with $i< k \leq j$ ($j<k\leq i$ is similar), we have $\theta_{k,i} = O(r^k)$ and $\theta_{k,j}\leq 0 $ for $r$ small enough, and hence,
	\begin{align*}
		&\abs{\E_{\pi}\left[\epsilon_{H,\theta^\uu_k,ij}\left( Z^\uu \right)\tilde{f}_{\theta^{(r)}_k}\left( Z^\uu \right)\right]} \\
		&  = \abs{\E_{\pi}\left[\left( \theta_{k,i}^\uu \theta_{k,j}^\uu \left(\dot{\kappa}\left(\theta_{k,i}^\uu Z^\uu_i\right)\dot{\kappa}\left(\theta^\uu_{k,j} Z^\uu_j\right) - 1\right)+\left(\theta_{k,i}^\uu\right)^2\ddot{\kappa}\left(\theta_{k,i}^\uu Z^\uu_i\right)\I{i=j} \right)\tilde{f}_{\theta^{(r)}_k}\left( Z^\uu \right)\right]} \\
		&  = \abs{\E_{\pi}\left[ \theta_{k,i}^\uu \theta_{k,j}^\uu \left(\dot{\kappa}\left(\theta_{k,i}^\uu Z^\uu_i\right) - 1\right) \tilde{f}_{\theta^{(r)}_k}\left( Z^\uu \right)\right]}\leq \abs{\theta_{k,i}^\uu \theta_{k,j}^\uu} \E_{\pi}\left[ \abs{  \dot{\kappa}\left(\theta_{k,i}^\uu Z^\uu_i\right) - 1 }  \abs{ \tilde{f}_{\theta^{(r)}_k}\left( Z^\uu \right)}\right]\\
		& \leq  \abs{\theta_{k,i}^\uu \theta_{k,j}^\uu} \E_{\pi}\left[  \I{\abs{\theta_{k,i}^\uu Z_i^\uu}>1} \right] e^{2d} = O(r^k) O(r^{j}) O(r^d) O(1) = o(r^{2k}),
	\end{align*}
	where the first inequality follows from \eqref{eq:kappa properties} and the third equality holds due to \eqref{eq: eps}. 

	The proof of $\gamma^\uu(\tilde{\theta}^{(r)}_k) = o(r^{2k+1/2})$ is similar to the proof of $\gamma^\uu(\theta^{(r)}_k) = o(r^{2k})$, and hence, we omit the detail.
\end{proof}

\end{document}